\documentclass[leqno]{article}

\usepackage{natbib}
\usepackage{geometry}
\usepackage{ts_def_share}

\usepackage{enumerate}
\usepackage{authblk}
\numberwithin{equation}{section}

\newcommand{\cplxtrans}{\dagger}

\begin{document}

\title{\textbf{Affine Volterra covariance processes and application to commodity markets}}

\author[1]{Boris G\"unther}

\author[1]{Ludger Overbeck}

\affil[1]{\footnotesize Institute of Mathematics, University of Gie{\ss}en, Germany}

\date{\today}
\maketitle

\begin{abstract}
	We study affine stochastic Volterra equations on the cone of symmetric positive semidefinite matrices. For scalar kernels acting entrywise on the matrix dynamics, we establish weak existence by exploiting stochastic invariance results for Volterra equations on convex domains and derive a conditional Fourier--Laplace transform formula characterized by matrix-valued Riccati--Volterra equations.
    As an application, we extend the Gibson--Schwartz commodity model by replacing its variance-covariance structure with a Volterra--Wishart process. The resulting model allows for memory in the variances and for stochastic instantaneous correlation, while retaining affine tractability. Its joint Fourier--Laplace transform admits an exponential-affine representation governed by a matrix Riccati--Volterra equation. While the existence theory considered here excludes kernels that are singular at the origin, shifted fractional kernels remain admissible and provide a tractable specification with power-law memory.
	
	\medskip

	\noindent\textbf{Keywords:} Stochastic Volterra equations, affine processes, Wishart process, commodity markets, convenience yield, multivariate stochastic volatility.
	\medskip

	\noindent\textbf{2020 Mathematics Subject Classification:} 60H20, 45D05, 91G30.
\end{abstract}

\section{Introduction}
This paper studies a stochastic Volterra equation approach to affine processes on the cone of positive semidefinite symmetric $ d \times d$ matrices $\bS_d^+$, with an application to commodity markets.

Matrix-valued stochastic processes on $\bS_d^+$ provide a natural generalization of variance processes to the multivariate case and additionally allow for modeling of stochastic covariances.
Affine models constitute one of the most widely used model class in finance since they allow for rich dynamics while remaining analytically tractable. See also the seminal paper \cite{duffie_Affine_2003} for a first systematic treatment of
affine processes in a financial context.
The tractability of Markovian affine processes stems from the fact that their Fourier--Laplace transform has an exponential-affine dependence on the state, characterized by Riccati equations.
Such processes enjoy widespread applications in finance, for example in multi-asset option pricing, portfolio and risk
management and more recently also in commodity markets.

For the multivariate case a large part of the literature considers the following affine dynamics for a process
$X$ on $\bS_d^+$:
\begin{equation}
	dX_t = (b + MX_t + X_t M^\top)dt + \sqrt{X_t} dW_t Q + Q^\top dW_t^\top \sqrt{X}, \quad X_0 \in \bS_d^+,
\end{equation}
where $b$ is a suitably chosen matrix in $\bS_d^+$, $M, Q$ are matrices, and $W$ is a standard $d \times d$-matrix
of Brownian motions. Processes of this type were first studied by \cite{bru_Wishart_1991} and are called Wishart processes.
The main reason for the analytical tractability of this model is that the following exponential-affine transform formula holds:
\begin{equation}\label{intro:wishart-laplace}
	E \bigl[ \exp(- \Tr(uX_T)) \mid \cF_t \bigr] = \exp \bigl(- \phi(T-t,u) - \Tr(\psi(T-t,u)X_t) \bigr),
\end{equation}
for all $t \in [0,T]$ and $u \in \bS_d^+$. The functions $\phi$ and $\psi$ solve a system of deterministic Riccati equations determined by the model parameters.

Wishart processes arise naturally by squaring particular Ornstein--Uhlenbeck processes. The broader framework of affine processes on $\bS_d^+$, as well as on general symmetric cones, was systematically developed in \cite{cuchiero_Affine_2011,cuchiero_Affine_2016}, allowing in particular for more general drift dynamics than in the Wishart setting.

More recently, there has been increasing interest in tractable models that go beyond the Markovian setting, in particular following \cite{gatheral_Volatility_2018}. They find that log-volatility exhibits rough behavior, with trajectories less regular than those generated by classical semimartingale models.
Empirical studies on electricity markets \cite{bennedsen_rough_2017} and on commodity markets \cite{alfeus_Forecasting_2022, alfeus_Dynamic_2023,alfeus_Forecasting_2020,daluiso_Rough_2026} also highlight this behavior.
Models based on fractional kernels, which go back to the construction of fractional Brownian motion in \cite{mandelbrot_Fractional_1968}, have been proposed to this end to better capture such volatility behavior. In particular, when the Hurst parameter $H < \tfrac{1}{2}$
these models reflect the observed rough and irregular behavior of volatility.

Applying such kernels to the classical affine models gives rise to affine Volterra processes.
The theory of affine Volterra processes was recently developed by \cite{abijaber_Affine_2019} in the homogeneous case and extended by \cite{ackermann_Inhomogeneous_2022} to the inhomogeneous case. Affine Volterra processes on $\R^d$ are defined as solutions of stochastic convolution equations
\begin{equation}
	X_t = \int_{0}^{t} K(t-s) b(X_s) \, ds + \int_{0}^{t} K(t-s)\sigma(X_s) \, dW_s,
\end{equation}
with affine coefficients
\begin{align*}
	a(x) = \sigma(x)\sigma(x)^\top & = A^0 + x_1A^1 + \dotsm + x_d A^d, \\
	b(x)                           & = b^0 + x_1b^1 + \dotsm + x_d b^d,
\end{align*}
for symmetric matrices $A^i$ and vectors $b^i$, and with kernels $K$ that may be singular, such as fractional kernels. Although affine Volterra processes are generally non-Markovian and in general need not be semimartingales, they preserve an exponential-affine transform formula in terms of Riccati--Volterra equations. Importantly, the transform formula remains affine with respect to the past path, making them tractable models for volatility models beyond the Markovian setting.

In this paper we study affine Volterra processes on $\bS_d^+$ defined by stochastic Volterra equations of the form
\begin{align}\label{eq:Volterra--Wishart-intro}
	\begin{split}
		X_t & = X_0 + \int_{0}^{t} K(t-s) b(X_s) \, ds + \int_{0}^{t} K(t-s) \Bigl( \sqrt{X_s}dW_s Q + Q^\top dW_s^\top \sqrt{X_s} \Bigr), \quad X_0 \in \bS_d^+.
	\end{split}
\end{align}
We show that, similar to the vector-valued Volterra case, this class remains affine and therefore provides a natural candidate for tractable stochastic volatility and covariance modeling.
We also give sufficient conditions on the parameters and the kernel that ensure existence of an $\bS_d^+$-valued weak solution, using existing viability results for stochastic Volterra equations on convex cones, see \cite{alfonsi_Nonnegativity_2025,abijaber_Weak_2026}. This is an essential point for the direct stochastic Volterra equation approach, since the positive semidefinite constraint is not built into the construction.
When $K$ is the identity matrix, the model reduces to the classical affine matrix diffusion studied in \cite{cuchiero_Affine_2011}.

By identifying $\bS_d$ with a Euclidean space, the affine Volterra framework developed in \cite{abijaber_Affine_2019,ackermann_Inhomogeneous_2022} can in principle be applied to matrix-valued processes. We nevertheless formulate the results directly on $\bS_d^+$, since this preserves the cone geometry and the algebraic structure of Wishart-type coefficients, and yields matrix Riccati--Volterra equations directly in their natural form.

Specifying $X$ as an $\bS_d^+$-valued process is essential for its interpretation as a covariance process. At the same time, the model should be flexible enough to capture more general dynamics and, in particular, memory effects in the covariance process.
In the matrix-valued setting, such memory effects are not restricted to the individual variances, but also enter the covariance structure. The model can therefore describe persistent dependence in both marginal volatilities and their co-movement.
A restriction of the proposed stochastic Volterra equation \eqref{eq:Volterra--Wishart-intro}, as already highlighted in \cite{abijaber_Weak_2026}, is that our sufficient conditions do not cover kernels that are singular at the origin, and hence excludes the fractional kernels commonly used in rough volatility models. In the applications, we therefore consider shifted fractional kernels. These regularize the singularity at the origin while retaining the characteristic power-law decay of fractional kernels away from the origin, thereby allowing us to incorporate power-law memory within the admissible class of kernels.

The main focus of this paper is therefore not the construction of genuinely rough matrix-valued Volterra processes, but rather the associated matrix transform theory, the analysis of the corresponding Riccati--Volterra equations, and their application to commodity markets.

The Volterra--Wishart process is the particular subcase in which the drift of \eqref{eq:Volterra--Wishart-intro} matches that of the classical Wishart process. Other definitions of Volterra--Wishart processes have already been proposed in the literature. \cite{jaber_Laplace_2022} defines a Volterra--Wishart process as the square $XX^\top$, where $X$ is a matrix-valued Gaussian Volterra process of the form
$$X_t = g(t) + \int_{0}^{t} K(t-s) \, dW_s.$$
\cite{cuchiero_Markovian_2019} construct the Volterra--Wishart processes from squares of infinite-dimensional Ornstein--Uhlenbeck processes, whose state space consists of matrix-valued measures.

The Volterra--Wishart process considered here is instead defined directly as the solution of a matrix-valued stochastic Volterra equation and therefore extends beyond quadratic constructions of Gaussian processes. This formulation allows for affine drift structures analogous to the classical Wishart process and provides a flexible non-Markovian covariance model. In contrast to the quadratic constructions, however, positive semidefiniteness is no longer automatic and must be established directly. Under the invariance conditions used in this paper, this direct construction requires kernels that are finite at the origin, whereas the quadratic constructions above can accommodate singular kernels. The two approaches therefore provide different advantages: quadratic constructions naturally preserve positive semidefiniteness and allow for singular kernels, while the direct stochastic Volterra equation formulation gives more flexibility in specifying affine matrix dynamics.

As an application we consider a popular commodity model and extend its volatility term by a Volterra--Wishart variance-covariance process.
This application is inspired by the recent work of \cite{schneider_Revisiting_2024}, where they extend the classical two-factor \cite{gibson_Stochastic_1990} and \cite{schwartz_ShortTerm_2000} commodity models by a Wishart variance-covariance process. Their extension introduces stochastic volatility and stochastic correlation while preserving analytical tractability through exponential-affine characteristic functions, which they use for derivative pricing and the construction of state-space representations for model estimation.
In these commodity models, future prices are described by the joint stochastic evolution of the spot price and convenience yield over time with constant volatility parameters. The convenience yield models the benefit of immediate ownership of a physical commodity. This benefit arises because holding inventories allows firms to respond quickly and efficiently to unexpected changes in supply or demand.
The convenience yield follows a mean-reverting, Ornstein--Uhlenbeck process, while the drift of the spot price is given by the short rate minus the convenience yield. These models fall under the same affine framework as studied by \cite{duffie_Affine_2003}, which explains their tractability and popularity. There is strong empirical evidence that volatility should be time-varying, e.g. because of volatility clustering or because the implied volatility surface cannot be matched by constant volatility specifications. Consequently, several articles have proposed three-factor commodity models with stochastic volatility \cite{deng_Stochastic_,hikspoors_Asymptotic_2008,hughen_maximal_2010,chen_Pricing_2017}, notably including \cite{richter_Stochastic_2002}, where they introduce a commodity model with stochastic volatility \`a la Heston.

In \cite{schneider_Revisiting_2024} they go one step further and extend the Gibson--Schwartz and Schwartz--Smith models by a Wishart variance-covariance process. They show that these models, now extended by a matrix-valued affine stochastic volatility, retain their tractability through an exponential-affine transform representation analogous to \eqref{intro:wishart-laplace}.
This extension is motivated by the fact that, according to \cite{routledge_Equilibrium_2000}, the relationship between the spot price and the convenience yield is not static but changes over time, reflecting shifts in underlying market conditions and inventory levels.
Empirical studies, such as \cite{omura_Convenience_2015}, demonstrate that spot prices and convenience yields exhibit stronger positive co-movement when inventories are low, consistent with periods of backwardation, while the correlation weakens when inventories are more abundant, in line with storage theory.
Moreover, \cite{pindyck_Volatility_2004} highlights that periods of heightened volatility are typically associated with simultaneous increases in both the spot price and the convenience yield.
In other words, higher uncertainty tends to raise commodity prices and increase the benefit of holding inventories in the short term.
This naturally suggests a time-varying instantaneous correlation structure. Taken together, these findings provide strong motivation for employing a Wishart process to describe the variance-covariance structure, as it accommodates both stochastic volatility and time-varying correlations in a coherent multivariate setting.

From an economic perspective, a Wishart specification also has strong intuitive appeal.
In modern electricity markets with a high penetration of renewable energy, there are often multiple sources of uncertainty acting simultaneously: weather-dependent renewable output, fuel price shocks, demand fluctuations, and regulatory interventions.
These factors do not affect the spot price and the convenience yield in isolation but rather generate complex patterns of co-movement and volatility spillovers across the two processes.
The concept of volatility spillovers was first introduced in the context of financial markets by \cite{diebold_Measuring_2009, diebold_Better_2012} and has since been further explored for commodity markets by \cite{barbaglia_Volatility_2020, ma_Does_2022, schischke_Impact_2025}.
For example, a sudden drop in wind generation can simultaneously cause a spike in the spot price and increase the value of reserve generation capacity, which is directly linked to the convenience yield.
Such joint effects are naturally captured by a multivariate stochastic volatility model such as the Wishart process.

The Volterra extension adds another dimension to this modeling framework. In addition to allowing variances and correlations to evolve stochastically, it permits shocks to the variance-covariance structure to have persistent effects on future market dynamics. This is particularly natural in commodity markets, where volatility and dependence structures may reflect persistent effects of inventory conditions, supply disruptions, or other market-specific shocks.
In line with the idea of \cite{schneider_Revisiting_2024}, we extend the Gibson--Schwartz model by a Volterra--Wishart variance-covariance process and show that the corresponding joint Fourier--Laplace transform remains exponential-affine and is characterized by matrix Riccati--Volterra equations.

\medskip
The remainder of the paper is structured as follows. Section \ref{section:affine-volterra} defines and studies the main process considered in this paper. In Section \ref{section:Fourier--Laplace}, we establish the exponential-affine transform formula and derive its path-dependent representation. Finally, Section \ref{section:Commodity-models} extends the Gibson--Schwartz commodity model with a Volterra--Wishart variance-covariance process and shows that their joint Fourier--Laplace transform retains an exponential-affine representation.

\subsection{Notation and preliminaries}

Here, we introduce the notation used throughout and recalls key properties of matrix-valued equations.
We begin by fixing some conventions.

Throughout the paper, we view elements of $\R^m$ and $\C^m = \R^m + i\R^m$ as column vectors, while elements of the dual spaces $(\R^m)^\ast$ and $(\C^m)^\ast$ are viewed as row vectors.
For every $m,n\in \N$ we denote by $\lVert \cdot\rVert$ the Frobenius norm $\lVert \cdot \rVert \colon \C^{m\times n} \to [0,\infty)$ by
$$\lVert A\rVert^2= \sum_{k = 1}^m \sum_{j = 1}^n (\re A_{kj})^2 + (\im A_{kj})^2,$$
where $A = (A_{kj})_{k =1,\ldots,m; j = 1,\ldots,n} \in \C^{m\times n}$.
For $m,n \in \N$ and $A = (A_{kj})_{k =1,\ldots,m; j = 1,\ldots,n} \in \C^{m \times n}$ we denote by $A^\top = (A_{jk})_{j =1,\ldots,n; k = 1,\ldots,m} \in \C^{n\times m}$ the transpose of $A$.
For $T \in (0,\infty)$, $m,n \in\N$, $f\colon [0,T] \to \C^{n\times m}$ and $h \in [0,T]$ the shifted function $\Delta_h f\colon [0,T] \to \C^{n \times m}$ satisfies $\Delta_h f(t) = f((t+h)\wedge T)$ for all $t \in [0,T]$.
We denote for $T \in (0,\infty)$, $m,n \in \N$ and $f\colon [0,T] \to \C^{n \times m}$ the total variation of $f$ over $[0,T]$ by $$\| f \|_{TV} = \sup \sum_{j} \| f(t_{j+1}) - f(t_j) \|$$
where the supremum is taken over all partitions $0\leq t_1 <\ldots < t_N \leq T$, $N \in \N$.
For a fixed dimension $d \in \N$, we denote by $\bS_d$ the space of symmetric real-valued $d \times d$ matrices, equipped with the scalar product $\langle x, y \rangle = \Tr(xy)$ for $x, y \in \bS_d$. When considering complex-valued ($i\R$-valued) entries, we write $\bS_d(\C)$ ($\bS_d(i\R)$).
The space $\bS_d$ is isomorphic, though not isometric, to $\R^{d(d+1)/2}$. A standard basis of $\bS_d$ is given by $\{c^{ij}, i \leq j\}$, where the $(kl)$-th component of $c^{ij}$ is $c^{ij}_{kl} = \delta_{ik}\delta_{jl} + \delta_{jk}\delta_{il}(1-\delta_{ij})$, and $\delta_{ij}$ denotes the Kronecker delta. We write $\bS_d^+$ for the cone of symmetric positive semidefinite matrices. Recall that $\bS_d^+$ is self-dual with respect to the above scalar product, meaning that
\begin{equation*}
	\bS_d^+ = \{x \in \bS_d \mid \langle x, y \rangle \geq 0 \ \text{for all } y \in \bS_d^+\}.
\end{equation*}
The cone $\bS_d^+$ induces a partial order relation on $\bS_d$. We write $x \succeq y$ if $x-y \in \bS_d^+$.
Furthermore, we denote by $\GL_d$ the general linear group over $\R$, and by $\Id_d$ the $d \times d$ identity matrix.

Convolutions play a central role in the theory of Volterra processes. Let $K: \R_+ \to \R$ be a measurable function and $L$ a measure on $\R_+$ of locally bounded variation. The left and right convolutions of $K$ and $L$ are defined for $t > 0$ by
\begin{equation*}
	(K \ast L)(t) = \int_{[0,t]} K(t-s)\,L(ds),
	\qquad
	(L \ast K)(t) = \int_{[0,t]} L(ds)\,K(t-s),
\end{equation*}
whenever these expressions are well-defined, and are extended to $t = 0$ by right-continuity whenever possible. If $K$ and $L$ are matrix-valued, only one of these convolutions may be well-defined, depending on the dimensions of the matrices involved. When $F: \R_+ \to \R$ is a measurable function, we use the shorthand $K \ast F = K \ast (F dt)$, which gives
\begin{equation*}
	(K \ast F)(t) = \int_{0}^{t} K(t-s) F(s)\, ds.
\end{equation*}
In particular, if $K \in L^1(\R_+)$ and $F$ is continuous, then $K \ast F$ is also continuous. This notion naturally extends to stochastic processes.
Let $K$ be square integrable on $[0,T]$ and $Z$ a continuous $d \times d$ semimartingale. The stochastic convolution of $K$ with $Z$
\begin{equation}\label{def:ast-local-martingale}
	(K \ast dZ)(t) = \int_{[0,t]} K(t-s) \, dZ_s,
\end{equation}
is well-defined as an It\^o integral for any $t \geq 0$ provided that
\begin{equation*}
	\int_{0}^{t} |K(t-s)|^2 \, d\Tr \langle Z \rangle_s < \infty.
\end{equation*}
In particular, if $K \in L^2(\R_+)$ and the quadratic variation of $Z$ has the form $\langle Z \rangle_s = \int_{0}^{s} a_u \, du$ for some locally bounded process $a$, then the stochastic convolution \eqref{def:ast-local-martingale} is well-defined for every $t \geq 0$.

Two matrix operators will play an important role throughout this article: the trace and the vectorization operators. The trace operator is linear and invariant under cyclic permutations: for any square matrices $A,B,C,D$, we have
\begin{equation*}
	\Tr(ABCD) = \Tr(DABC) = \Tr(CDAB) = \Tr(BCDA).
\end{equation*}
The vectorization operator, denoted by $\vc(\cdot)$, maps a $d \times d$ matrix to a vector in $\C^{d^2}$ by stacking its columns. Equipped with the Frobenius inner product, this mapping is unitary in the sense that
\begin{equation*}
	\Tr(A^\cplxtrans B) = \vc(A)^\cplxtrans \vc(B),
\end{equation*}
where $\cplxtrans$ denotes the conjugate transpose. Moreover, vectorization provides a convenient way to express matrix multiplication as a linear transformation. For matrices $A, B, C$ of compatible dimensions, one has
\begin{equation*}
	\vc(ABC) = (C^\top \otimes A)\, \vc(B),
\end{equation*}
where $\otimes$ denotes the Kronecker product. The Kronecker product satisfies $(A \otimes B)^\top = A^\top \otimes B^\top$, and it allows us to rewrite matrix equations in a linear form:
\begin{equation*}
	ABC = D \quad \Longleftrightarrow \quad (C^\top \otimes A)\, \vc(B) = \vc(D).
\end{equation*}
This representation will be particularly useful in the calculations and proofs presented in the later sections. 

\section{Affine Volterra processes on positive semidefinite matrices}\label{section:affine-volterra}

In this section, we study affine Volterra processes whose state space is the cone of symmetric positive semidefinite matrices.
Furthermore, we establish several preparatory results that were previously proven by \cite{abijaber_Affine_2019} and \cite{ackermann_Inhomogeneous_2022} in the vector-valued case.
Since we are working in the matrix-valued setting, these results need to be adapted accordingly to fit this setting.

The canonical equation underlying this setting is
\begin{align}
	\begin{split}
		X_t = X_0 + \int_{0}^{t} K(t-s)\bigl( \beta + MX_s + X_s M^\top\bigr) ds + \int_{0}^{t} K(t-s) \bigl(\sqrt{X_s}dW_s Q + Q^\top dW_s^\top \sqrt{X_s} \bigr),
	\end{split}
\end{align}
for a convolution kernel $K \in \loc^2(\R_+, \R^{d \times d})$, some suitable chosen matrices $\beta,\, Q,\, M$ and a standard $d \times d$-matrix of Brownian motions $W$ with initial condition $X_0 \in \bS_d^+$. Whenever this equation admits an $\bS_d^+$-valued solution, we refer to the corresponding process as a Volterra--Wishart process.
When the kernel $K$ is chosen to be the identity matrix, the equation reduces to the classical Markovian Wishart SDE introduced by \cite{bru_Wishart_1991}, thereby recovering the standard affine framework as a special case.

Before stating the main definition of an $\bS_d^+$-valued affine Volterra process, we give two structural limitations that arise when one attempts to construct such processes directly from the stochastic Volterra equation by means of existing stochastic-invariance arguments. These restrictions should be understood as sufficient conditions, rather than as necessary conditions for the existence of an $\bS_d^+$-valued solution.

\begin{enumerate}
	\item[(i)]\emph{Restrictions on the matrix structure of the kernel.}
	
	The available boundary-invariance arguments in the literature can be applied directly when the kernel acts compatibly with the geometry of $\bS_d^+$. A convenient sufficient condition, which we impose in the Volterra--Wishart setting, is
	\begin{equation*}
	K(t) = k(t)\Id_d, \quad t \geq 0,
	\end{equation*}
	for some nonnegative scalar-valued kernel $k$. This restriction does not imply that more general matrix-valued kernels necessarily destroy positivity. Rather, for a general matrix-valued kernel, the usual tangency argument for the diffusion coefficient no longer applies directly.
Indeed, suppose that, at some time $\tau$, the process reaches the boundary of the cone, $X_\tau \in \partial \bS_d^+$. Then there exists a nonzero vector $v \in \ker(X_\tau)$, and hence $\sqrt{X_\tau} v = 0$. Stochastic invariance requires that the diffusion does not generate fluctuations in the outward normal direction corresponding to $vv^\top$, or equivalently that its contribution to the scalar quantity $v^\top Xv$ vanishes at the boundary. If $K(r) = k(r)\Id_d$, then, for every matrix $A \in \R^{d \times d}$,
\begin{equation*}
	v^\top K(r) \bigl( \sqrt{X_\tau}AQ + Q^\top A^\top\sqrt{X_\tau} \bigr)v = k(r)v^\top\bigl(\sqrt{X_\tau}AQ + Q^\top A^\top\sqrt{X_\tau} \bigr)v = 0
\end{equation*}
since $\sqrt{X_\tau}v = 0$. Thus, the diffusion is tangent to the boundary in the relevant directions.
For a general matrix-valued kernel $K(r)$, however, the corresponding expression contains the term
\begin{equation*}
	v^\top K(r)\sqrt{X_\tau}AQv = \bigl(\sqrt{X_\tau}K(r)^\top v\bigr)^\top AQv,
\end{equation*}
which does not necessarily vanish, since $K(r)^\top v$ does not in general belong to $\ker(X_\tau)$. Consequently, the classical boundary argument does not immediately yield stochastic invariance. Additional structural assumptions on $K$, or a different invariance argument, would be required to treat genuinely matrix-valued kernels.

	\item[(ii)] \emph{Boundedness of the kernel.}

	A second restriction arises from the sufficient inward-pointing condition on the drift. For the classical Wishart SDE, a standard condition ensuring invariance of $\bS_d^+$ is $\beta-(d-1)Q^\top Q \in \bS_d^+$, see \cite{bru_Wishart_1991}. 
	In the Volterra--Wishart setting with convolution kernel $K(t) = k(t)\Id_d$, the analogous invariance calculation leads to the sufficient condition
	$$\beta - (d-1)k(0)Q^\top Q \in \bS_d^+.$$	
	as already highlighted in \cite{abijaber_Weak_2026}. While the general invariance results for stochastic Volterra equations in \cite{abijaber_Weak_2026, alfonsi_Nonnegativity_2025} allow for unbounded kernels, the above condition imposes an additional restriction in the affine setting. In particular, it excludes singular kernels such as the fractional kernel
	\begin{equation*}
		k(t) = \frac{t^{\alpha-1}}{\Gamma(\alpha)}, \quad \alpha \in (0,1),
	\end{equation*}
for which $k(0) = \infty$. For this reason, in the concrete applications below, we work instead with shifted fractional kernels, which retain the qualitative long-memory behavior of the fractional kernel while remaining bounded at the origin.
\end{enumerate}

Inspired by the results in \cite{cuchiero_Affine_2011} and with those mentioned limitations, we extend the Volterra--Wishart model class introduced above by allowing for a more general, fully affine drift term. This leads us to the following definition, which specifies the class studied in this paper.

\begin{definition}\label{def:affin-volterra}
	Fix $d \in \N$ and $T < \infty$. Let $K \in \loc^2(\R_+, \R^{d \times d})$ and let $W$ be a $d \times d$ matrix Brownian motion, let $Q \in \R^{d \times d}$ and set $\alpha := Q^\top Q \in \bS_d^+$, which parametrizes the diffusion covariance.
	A continuous process $X = (X_t)_{t \in [0,T]}$ is an \emph{affine Volterra process} on $\bS_d^+$ if it is a $\bS_d^+$-valued weak solution to
	\begin{align}\label{eq:main-SDE}
		X_t & = X_0 + \int_{0}^{t} K(t-s) b(X_s)\, ds + \int_{0}^{t} K(t-s) \Bigl(\sqrt{X_s}\, dW_s\, Q
		+ Q^\top dW_s^\top \sqrt{X_s}\Bigr),
	\end{align}
	with deterministic $X_0 \in \bS_d^+$ and an affine drift map $b:\bS_d\to\bS_d$,
	\begin{equation*}
		b(x)= b_0 + B(x), \quad x\in\bS_d,
	\end{equation*}
	for some $b_0\in\bS_d$ and a linear map $B:\bS_d\to\bS_d$.
\end{definition}

We deliberately do not impose a particular matrix structure on the kernel $K$ in Definition~\ref{def:affin-volterra}. The requirement that the process $X$ is an $\bS_d^+$-valued weak solution already imposes an implicit compatibility condition between the kernel and the coefficients. For an arbitrary matrix-valued kernels, neither symmetry nor positive semidefiniteness of the right-hand side is automatic.
In our existence result, we impose sufficient conditions of this type by restricting to kernels of the form $K(t) = k(t)\Id_d$, where $k$ is scalar-valued. This restriction is convenient for applying the stochastic invariance arguments available in the literature and will be sufficient for all applications considered below.
To this end, we impose the following admissibility conditions on the kernel and the drift coefficient, in analogy with the classical affine Wishart setting of \cite{cuchiero_Affine_2011}.

\begin{assumption}\label{ass:admissible-parameters}
    Let
    \begin{equation*}
        K(t) = k(t)\Id_d, \quad t \in [0,T],
    \end{equation*}
    for a nonnegative bounded scalar-valued kernel $k$.
    We assume that
    \begin{equation}\label{eq:affine-admissibility}
		b_0 \succeq k(0)(d-1)\alpha, \quad \text{ and }\quad  \Tr\bigl(B(x)u\bigr) \geq 0 \text{ for all } x,u \in \bS_d^+ \text{ with } \Tr(xu) = 0 .
	\end{equation}
\end{assumption}

\begin{remark}
Since the kernel $K$ in Definition~\ref{def:affin-volterra} acts only from the left in the stochastic convolution, additional conditions on the kernel are needed to ensure that the solution remains in the state space $\bS_d^+$. In particular, for a general matrix-valued kernel, even symmetry of the right-hand side is not automatic.
A natural way to accommodate more general matrix-valued kernels while preserving the symmetric matrix structure is to let the kernel act from both sides. More precisely, for a kernel 
$K \in \loc^4(\R_+,\R^{d \times d})$, one may consider
\begin{align*}
		X_t & = X_0 + \int_{0}^{t} K(t-s) b(X_s) K(t-s)^\top \, ds + \int_{0}^{t} K(t-s) \Bigl(\sqrt{X_s}\, dW_s\, Q
		+ Q^\top dW_s^\top \sqrt{X_s}\Bigr) K(t-s)^\top.
	\end{align*}
This sandwich structure is naturally compatible with the cone $\bS_d^+$. In particular, it preserves symmetry by construction, since for every $A \in \bS_d$ and every $K \in \R^{d \times d}$, we have $KAK^\top \in \bS_d$.
Moreover, for every $B \in \bS_d^+$ and every $K \in \R^{d \times d}$, we similarly have $KxK^\top \in \bS_d^+$.
Thus, the action of the matrix-valued kernel itself preserves both symmetry and positive semidefiniteness, independently of any diagonal or commutativity structure of $K$.

We emphasize, however, that this observation alone does not solve the stochastic invariance problem for the resulting Volterra equation. In particular, although the sandwich action preserves positive semidefiniteness of positive semidefinite matrices, the stochastic increments in the martingale term are not themselves positive semidefinite. Hence, one still has to verify an appropriate boundary tangency or stochastic invariance condition ensuring that the full dynamics remain in $\bS_d^+$.

Conditional on the existence of such an $\bS_d^+$-valued solution, the vectorization and affine transform arguments developed below can be adapted
to this formulation as well. The corresponding equations then contain the additional multiplication by $K^\top$. Equivalently, after vectorization, the convolution kernel is given by the Kronecker product $K \otimes K$. Since our applications only require diagonal kernels, we do not pursue this more general formulation here.
\end{remark}

\begin{remark}\label{rem:Aijkl}
	Let $M = (M_t)_{t \in [0,T]}$ denote the $\bS_d$-valued local martingale defined by
	\begin{equation*}
		dM_t
		:= \sqrt{X_t}\, dW_t Q + Q^\top dW_t^\top \sqrt{X_t},
		\qquad \alpha := Q^\top Q \in \bS_d^+ .
	\end{equation*}
	Then, for all $i,j,k,l\in\{1,\dots,d\}$, the quadratic covariation of the scalar
	components of $M$ satisfies
	\begin{equation*}
		d\langle M_{ij}, M_{kl}\rangle_t
		= A_{ijkl}(X_t)\,dt,
	\end{equation*}
	with
	\begin{equation*}
		A_{ijkl}(x)
		:= x_{ik}\alpha_{jl}+x_{il}\alpha_{jk}
		+x_{jk}\alpha_{il}+x_{jl}\alpha_{ik},
		\qquad x\in\bS_d .
	\end{equation*}
	Consequently, the mapping $x\mapsto A(x)$ is \emph{linear} on $\bS_d$, and its linear part is of the specific form
	\begin{equation*}
		\langle u, A(x)u\rangle = 4\langle x, u \alpha u \rangle, \quad x,u \in \bS_d.
	\end{equation*}
	Thus, in the sense of affine matrix diffusions \cite{cuchiero_Affine_2011}, $\alpha$ parametrizes a \emph{linear diffusion coefficient} on $\bS_d^+$.
	Furthermore, since $\bS_d$ is a finite-dimensional vector space, any linear map
	$B:\bS_d\to\bS_d$ admits a coordinate representation of the form
	\begin{equation*}
		B(x) = \sum_{i,j = 1}^d \beta^{ij} x_{ij},
		\qquad x\in\bS_d,
	\end{equation*}
	for a family of matrices $\beta^{ij} = \beta^{ji}\in\bS_d$.
	With respect to the Hilbert--Schmidt inner product
	\begin{equation*}
		\langle x,u\rangle := \Tr(xu), \qquad x,u\in\bS_d,
	\end{equation*}
	the adjoint operator $B^\top:\bS_d\to\bS_d$ is defined by
	\begin{equation*}
		\langle B(x),u\rangle = \langle x, B^\top(u)\rangle,
		\qquad x,u\in\bS_d.
	\end{equation*}
	Inserting the above representation of $B$ yields the characterization
	\begin{equation*}
		(B^\top(u))_{ij} = \langle \beta^{ij}, u\rangle,
		\qquad i,j\in\{1,\dots,d\}.
	\end{equation*}
\end{remark}

The structure given by \eqref{eq:affine-admissibility} ensures that the process preserves the positive semidefinite property of $X_t$ over time, in analogy to the classical Wishart setting.
In fact, the assumptions imposed on the drift and the stochastic integrals are essentially the same as those required in the classical theory of matrix-valued stochastic processes on \(\bS_d^+\).

Under these conditions, the process given by Definition \ref{def:affin-volterra} generalizes the Volterra--Wishart process by allowing for a fully affine drift structure while preserving the natural positivity constraints of the state space.

\begin{remark}\label{remark:sde-with-sigma}
	As shown in \cite{cuchiero_Affine_2011}, the stochastic affine Volterra equation \eqref{eq:main-SDE} can be rewritten in a more compact form by collecting the diffusion terms. Specifically, we can express it as
	\begin{align}\label{eq:compact-form}
		X_t = X_0 + \int_{0}^{t} K(t-s) b(X_s) \, ds
		+ \int_{0}^{t} K(t-s) \Big( \sum_{k,l = 1}^d \sigma^{kl}(X_s) \, dW_{s,kl} \Big),
	\end{align}
	where the matrices $\sigma^{kl}(x) \in \bS_d$ are given by
	\begin{equation*}
		\sigma^{kl}(x) = \sqrt{x} \, M^{kl} Q + Q^\top M^{lk} \sqrt{x}, \quad x \in \bS_d^+,
	\end{equation*}
	and the matrices $M^{kl} \in \R^{d \times d}$ are defined componentwise by
	\begin{equation*}
		M_{ij}^{kl} = \delta_{ik} \delta_{jl}, \quad 1 \leq i,j,k,l \leq d.
	\end{equation*}
	Its quadratic covariation is thus given by
	\begin{equation*}
		A_{ijkl}(x) = \sum_{m,n = 1}^d \sigma_{ij}^{mn}(x) \, \sigma_{kl}^{mn}(x), \quad x \in \bS_d^+.
	\end{equation*}
	It is worth mentioning that the coefficients $b$ and $\sigma^{kl}$, for every $k,l \in \{1, \dotsc, d\}$, are by construction continuous and also satisfy the linear growth condition
	\begin{equation}\label{eq:LG-matrix}
		\|b(x)\|
		+ \Bigl(\sum_{k,l = 1}^d \|\sigma_{kl}(x)\|^2\Bigr)^{1/2}
		\leq c_{LG}\bigl(1+\|x\|\bigr), \quad x \in \bS_d^+,
	\end{equation}
	for a positive constant $c_{LG}$.
\end{remark}

If we apply the vectorization operator to the representation \eqref{eq:compact-form} of the $\bS_d^+$-valued affine Volterra process \(X\), we can transform it into a vector-valued Volterra process.
This allows us to directly apply known existing results established in \cite{abijaber_Affine_2019,ackermann_Inhomogeneous_2022}.
Thus, we obtain for the vectorization of \eqref{eq:compact-form}
\begin{align*}
	Y_t
	 & = \vc(X_t) = Y_0
	+ \int_{0}^{t} \bigl(\Id_d \otimes K(t-s)\bigr) \vc(b(X_s))\,ds
	+ \int_{0}^{t} \bigl(\Id_d \otimes K(t-s)\bigr)
	\Bigl(\sum_{k,l = 1}^d \vc(\sigma^{kl}(X_s))\,dW_{s,kl}\Bigr)
\end{align*}
for $Y_0 = \vc(X_0)$.
Define the drift $\widetilde{b}(y) := \vc(b(x))$ for $y = \vc(x)$, the diffusion matrix
\begin{equation*}
	\widetilde\sigma(y)
	:= \bigl[\vc(\sigma^{11}(x)),\dots,\vc(\sigma^{dd}(x))\bigr]
	\in \R^{d^2\times d^2},
\end{equation*}
and the $d^2$-dimensional Brownian motion satisfying $\vc(W) = \widetilde{W}$.
Then $Y = (Y_t)_{t \in [0,T]}$ satisfies the $\R^{d^2}$ Volterra equation with kernel $\widetilde{K} := \bigl(\Id_d \otimes K\bigr) \in \loc^2(\R_+, \R^{d^2 \times d^2})$
\begin{equation}\label{eq:Y-volterra}
	Y_t = Y_0 + \int_{0}^{t} \widetilde{K}(t-s)\widetilde{b}(Y_s)\,ds + \int_{0}^{t} \widetilde{K}(t-s)\widetilde\sigma(Y_s)\,d\widetilde{W}_s, \quad t \in [0,T],
\end{equation}
which is precisely of the type studied in \cite{abijaber_Affine_2019}, on the state space $\vc(\bS_d^+)$. In particular, we obtain for the quadratic variation $A_{ijkl}(x) = (\sigma(x)\sigma(x)^\top)_{ijkl}$. In the case that the kernel is diagonal of the form $K = k \Id_d$ this reduces further to the $\R^{d^2}$ Volterra equation with kernel $k \in \loc^2(\R_+, \R_+)$
\begin{equation}\label{eq:Y-volterra-under-assumption}
	Y_t = Y_0 + \int_{0}^{t} k(t-s)\widetilde{b}(Y_s)\,ds + \int_{0}^{t} k(t-s)\widetilde\sigma(Y_s)\,d\widetilde{W}_s, \quad t \in [0,T],
\end{equation}

\begin{lemma}\label{lem:moment_bound}
	Fix $T < \infty$ and $p \geq 1$. Consider a continuous $\bS_d^+$-valued affine Volterra process $X = (X_t)_{t \in [0,T]}$. Then
	\begin{equation}\label{eq:moment_bound_wishart}
		\sup_{t \in [0,T]}E\bigl[\|X_t\|^{p}\bigr]\leq c,
	\end{equation}
	for a constant $c\in(0,\infty)$, depending only on $T, d^2, \|X_0\|, c_{LG}, \widetilde{K}$ and $p$, where $\widetilde{K} = \bigl(\Id_d \otimes K\bigr)$.
\end{lemma}

\begin{proof}
	Throughout the proof let $p \in (2,\infty)$. The case $p \in [1,2]$ then follows from Jensen's inequality.
	We argue by reduction to the vector-valued case treated in \cite[Lemma 3.1]{abijaber_Affine_2019}.
	Using $\eqref{eq:Y-volterra}$ we know that $Y := \vc(X)$ can be expressed as
	\begin{equation*}
		Y_t = Y_0 + \int_{0}^{t} \widetilde{K}(t-s)\widetilde{b}(Y_s)\,ds + \int_{0}^{t} \widetilde{K}(t-s)\widetilde\sigma(Y_s)\,d\widetilde{W}_s, \quad t \in [0,T].
	\end{equation*}
	It remains to verify the linear growth assumption (3.1) of \cite[Lemma 3.1]{abijaber_Affine_2019}.
	For $y = \vc(x)$ we have
	\begin{equation*}
		\|\widetilde{b}(y)\| = \|\vc(b(x))\| = \|b(x)\|, \quad
		\|\widetilde\sigma(y)\|^2 = \sum_{k,l = 1}^d \|\vc(\sigma_{kl}(x))\|^2 = \sum_{k,l = 1}^d \|\sigma_{kl}(x)\|^2.
	\end{equation*}
	Using \eqref{eq:LG-matrix} and $\|y\| = \|x\|$,
	\begin{equation*}
		\|\widetilde{b}(y)\| + \|\widetilde\sigma(y)\| \leq c_{LG}\bigl(1+\|x\|\bigr) = c_{LG}\bigl(1+\|y\|\bigr),
	\end{equation*}
	which verifies the linear growth assumption. Continuity follows from continuity of $b$ and $\sigma_{kl}$.
	Hence, we can apply \cite[Lemma 3.1]{abijaber_Affine_2019} to \eqref{eq:Y-volterra}
	with dimension $d^2$, kernel $\widetilde{K}$, and linear growth constant $c_{LG}$.
	Therefore, for every $p \geq 2$ there exists $c<\infty$ depending only on $T, d^2, \|X_0\|, c_{LG}, \widetilde{K}$ and $p$
	such that
	\begin{equation*}
		\sup_{t \in [0,T]}E[\|Y_t\|^p] = \sup_{t \in [0,T]}E[\|X_t\|^p] \leq c,
	\end{equation*}
	since $\|Y_t\| = \|X_t\|$.
\end{proof}

The following lemma is a simple consequence of the representation $\eqref{eq:main-SDE}$ and will be used in the proof of Theorem \ref{thm:main-thm}.

\begin{lemma}\label{lm:cond-expec}
	Let $X$ be an affine Volterra process on $\bS_d^+$. Then for all $s,t \in [0,T]$ we have that
	\begin{align*}
		E[X_s \mid \cF_t] & = X_0 + \int_{0}^{s} K(s-r) b(E[X_r \mid \cF_t]) dr                                           \\
		                  & + \int_{0}^{t \wedge s} K(s-r) \left(\sqrt{X_r}dW_r Q + Q^\top (dW_r)^\top \sqrt{X_r}\right).
	\end{align*}
\end{lemma}
\begin{proof}
	If $s < t$ we have by \eqref{eq:main-SDE}
	\begin{align*}
		E[X_s \mid \cF_t] & = X_s = X_0 + \int_{0}^{s} K(s-r) b(E[X_r \mid \cF_t]) dr + \int_{0}^{t \wedge s} K(s-r) \sqrt{X_r}dW_r Q \\
		                  & \quad + \int_{0}^{t \wedge s} K(s-r)Q^\top (dW_r)^\top \sqrt{X_r}.
	\end{align*}
	For every $s \in [0,T]$ let $M^s = (M_t^s)_{t \in [0,T]}$ be the process defined by
	\begin{equation*}
		M_t^s := \int_0^{t \wedge s} K(s-r)\sqrt{X_r}\,dW_r Q, \quad t \in [0,T].
	\end{equation*}
	Then, by vectorization,
	\begin{equation*}
		\vc(M_t^s)
		=
		\int_0^{t\wedge s}
		\big(Q^\top \otimes K(s-r)\sqrt{X_r}\big)\,d\vc(W_r).
	\end{equation*}
	Hence,
	\begin{align*}
		\Tr\big(\langle \vc(M^s) \rangle_T\big)
		 &= \int_{0}^{T \wedge s} \big\|Q^\top \otimes K(s-r)\sqrt{X_r}\big\|^2\,dr \\
		 &= \int_{0}^{T \wedge s} \|Q\|^2\,\|K(s-r)\sqrt{X_r}\|^2\,dr .
	\end{align*}
	Moreover,
	\begin{equation*}
		\|K(s-r)\sqrt{X_r}\|^2
		= \Tr\big(K(s-r)X_rK(s-r)^\top\big)
		\leq \|K(s-r)\|^2\Tr(X_r).
	\end{equation*}
	Consequently,
	\begin{equation*}
		\Tr \big(\langle \vc(M^s)\rangle_T\big)
		\leq \|Q\|^2\int_0^{T \wedge s} \|K(s-r)\|^2\,\Tr(X_r)\,dr,
	\end{equation*}
	and therefore, by Lemma \ref{lem:moment_bound} and since $K \in \loc^2(\R_+,\R^{d \times d})$,
	\begin{equation*}
		E\big[\Tr(\langle \vc(M^s)\rangle_T)\big] \leq \|Q\|^2 \Big(\sup_{r \in [0,T]} E[\Tr(X_r)]\Big) \int_{0}^{s} \|K(s-r)\|^2\,dr <\infty.
	\end{equation*}
	Thus, $M^s$ is a square-integrable martingale and similarly for
	\begin{equation*}
		\widetilde{M}_t^s := \int_0^{t\wedge s} K(s-r)Q^\top(dW_r)^\top\sqrt{X_r}.
	\end{equation*}

	Hence, in the case $s \geq t$ it holds by \eqref{eq:main-SDE} that
	\begin{align*}
		E[X_s \mid \cF_t] & = X_0 + E\left[\int_{0}^{s} K(s-r) b(X_r ) dr\mid \cF_t \right] + E\left[\int_{0}^{s} K(s-r) \sqrt{X_r}dW_r Q \mid \cF_t \right] \\
		                  & \quad + E\left[\int_{0}^{s} K(s-r)Q^\top (dW_r)^\top \sqrt{X_r}\mid \cF_t \right]                                                \\
		                  & = X_0 + \int_{0}^{s} K(s-r) b(E[X_r \mid \cF_t]) dr + \int_{0}^{t \wedge s} K(s-r) \sqrt{X_r}dW_r Q                  \\
		                  & \quad + \int_{0}^{t \wedge s} K(s-r)Q^\top (dW_r)^\top \sqrt{X_r}.
	\end{align*}
\end{proof}

The following condition expresses that the memory induced by the kernel does not destroy nonnegativity of previously accumulated contributions. It will therefore play a central role in the invariance arguments below.
\begin{definition}\label{def:nonnegativity-preserving-kernel}
	A nonnegative kernel $k: \R_+ \to \R_+$, such that $k(0) > 0$, is said to \emph{preserve nonnegativity} if for every $0 = t_0 < t_1 < \cdots < t_n$ and every $x_0, \ldots , x_n \in \R$ satisfying
\begin{equation*}
\sum_{j = 0}^{i} x_j k(t_i-t_j) \geq 0 , \quad i = 0,\ldots,n,
\end{equation*}
it follows that
\begin{equation*}
\sum_{t_j\leq t} x_j k(t-t_j) \geq 0, \quad t \geq 0.
\end{equation*}
\end{definition}

For the existence result we will cite the results of \cite{abijaber_Weak_2026, alfonsi_Nonnegativity_2025} which establish existence results of Volterra processes on convex domains. In \cite{alfonsi_Nonnegativity_2025} the existence is shown under Lipschitz assumptions on the coefficients under convolution kernels that preserve nonnegativity. \cite{abijaber_Weak_2026} generalizes these results in the sense that they consider also non-convolution kernels and only under linear growth conditions on the coefficients.

\begin{remark}
The assumptions on the kernel are satisfied by a large class of kernels arising in applications. In particular, every completely monotone kernel preserves nonnegativity \cite[Theorem 2.11]{alfonsi_Nonnegativity_2025}. Moreover, every completely monotone kernel is nonnegative and nonincreasing. Consequently, completely monotone kernels satisfy all structural assumptions of Assumption~\ref{ass:k-PSD-conditions}, except for the condition $k(0) > 0$, which has to be verified separately.
\end{remark}

Next, we impose regularity assumptions on the kernel $K$, which is the analogue to \cite[Condition (2.5)]{abijaber_Affine_2019}. But notice that since we only consider bounded kernels the sufficient assumption such that there exists a version of the convolution \eqref{def:ast-local-martingale} that is continuous in $t$ is the following.
\begin{assumption}\label{K-regularity-condition}
	Let $K = (K_{ij})_{i,j = 1,\dotsc,d} \in \loc^2(\R_+, \R^{d \times d})$. There exists $\gamma \in (0,1]$ and $c \in (0, \infty)$ such that for all $i,j \in \{1,\dotsc,d\}$ $\int_{s}^{t} |K_{ij}(t-u)|^2 \, du \leq c(t-s)^{\gamma}$ and$\int_{0}^{s} \left(K_{ij}(t-u)-K_{ij}(s-u)\right)^2 du \leq c(t-s)^\gamma$ for all
	$0 \leq s < t \leq T < \infty$.
\end{assumption}

For a continuous scalar kernel with $k(0) > 0$, the first integral is asymptotic to $k(0)^2 (t-s)$ as $t-s \downarrow 0$. Thus, an exponent of $\gamma > 1$ is unavailable in this setting. A continuously differentiable kernel satisfies the condition with $\gamma = 1$ on each compact interval.
The following result establishes weak existence of a solution to a
globally defined modification of \eqref{eq:main-SDE}. More precisely,
the matrix square root is replaced by the projection
$\sqrt{(\cdot)^+}$ onto $\bS_d^+$.

\begin{theorem}\label{thm:weak-exsitence-of-R-solutions}
    Let Assumption~\ref{ass:admissible-parameters} and
    Assumption~\ref{K-regularity-condition} hold. Then, for every
    $X_0\in\bS_d$, there exists a continuous $\bS_d$-valued weak
    solution to the modified stochastic Volterra equation
	\begin{equation*}
		X_t = X_0 + \int_{0}^{t} k(t-s) b(X_s) \,ds + \int_{0}^{t} k(t-s)\bigl(\sqrt{X_s^+} dW_s Q + Q^\top dW_s^\top \sqrt{X_s^+}\bigr), \quad t \in [0,T],
	\end{equation*}
    where $X^+$ denotes the orthogonal projection of $X \in \bS_d$ onto $\bS_d^+$.
    Moreover, the paths of $X$ are Hölder continuous of every order
    strictly smaller than $\gamma/2$, where $\gamma$ is the constant associated with $k$ and $T$ in Assumption~\ref{K-regularity-condition}.
\end{theorem}

\begin{proof}
	This holds true by \cite[Theorem 2.5]{abijaber_Weak_2026}.
\end{proof}

It remains to show that, whenever the
initial value belongs to $\bS_d^+$, the solution never leaves the
cone. To this end we verify the assumptions of the stochastic
invariance theorem of \cite{abijaber_Weak_2026}.
A central assumption of this theorem concerns the stochastic invariance
of an auxiliary finite-dimensional stochastic differential equation.
More precisely, for every $\lambda>0$ and $x\in\bS_d^+$, consider the auxiliary stochastic differential equation
\begin{equation}\label{eq:auxiliary-wishart-sde}
\xi_t^{\lambda,x} = x + \int_{0}^{t} \lambda b(\xi_s^{\lambda,x}) \, ds + \int_{0}^{t} \lambda \biggl(\sqrt{\xi_s^{\lambda,x}}\,dW_sQ + Q^\top dW_s^\top\sqrt{\xi_s^{\lambda,x}} \biggr),
\end{equation}
where the function $b$ is specified as in Definition~\ref{def:affin-volterra}.
In \cite{abijaber_Weak_2026}, this auxiliary SDE is required for every $\lambda \in \{K(t,t):t \in [0,T]\}$.
Since we consider convolution kernels, it suffices to verify the auxiliary SDE condition for
$\lambda = k(0)$.

We say that SDE$_{\lambda}(\bS_d^+)$ holds if for every $x \in \bS_d^+$, there exists a weak solution $\xi^{\lambda,x}$ to \eqref{eq:auxiliary-wishart-sde} such that $\P\bigl(\xi_t^{\lambda,x} \in \bS_d^+ \text{ for all }t \geq 0 \bigr) = 1$.

Furthermore, we need the following condition on the kernel.
\begin{assumption}\label{ass:k-PSD-conditions}
    Let Assumption~\ref{ass:admissible-parameters} hold. We additionally
    assume that the scalar kernel $k$ satisfies $k(0)>0$, is continuous, nonincreasing, and preserves nonnegativity.
\end{assumption}

\begin{theorem}\label{thm:weak-solutions}
	Let $X_0 \in \bS_d^+$ and let Assumption~\ref{ass:k-PSD-conditions} and Assumption~\ref{K-regularity-condition} hold true. Then \eqref{eq:main-SDE} admits a continuous
    $\bS_d^+$-valued weak solution.
\end{theorem}
\begin{proof}
	We apply \cite[Theorem 2.14]{abijaber_Weak_2026}. Since the kernel is of the form $K(t) = k(t)\Id_d$, the only point that remains to be checked is the invariance of the auxiliary SDE \eqref{eq:auxiliary-wishart-sde} for $\lambda = k(0)$.
	For this fixed $\lambda$, the auxiliary equation is an affine matrix diffusion with drift $\lambda b_0+\lambda B(x)$ and Wishart diffusion parameter $\lambda Q$, hence covariance parameter $\lambda^2\alpha$. By the stochastic invariance criterion for affine diffusions on $\bS_d^+$ in \cite[Theorem 2.4]{cuchiero_Affine_2011}, it is enough to verify
	\begin{equation*}
	   \lambda b_0 - (d-1) \lambda^2 \alpha \succeq 0 
	\end{equation*}
	and
	\begin{equation*}
	  \Tr(B(x)u) \geq 0 \text{ for all } x, u \in \bS_d^+ \text{ with }\Tr(xu) = 0.  
	\end{equation*}
	The second condition is exactly the second admissibility condition in \eqref{eq:affine-admissibility}. For the first one, Assumption~\ref{ass:k-PSD-conditions} implies that $k$ is nonincreasing, and therefore $k(0) = \sup_{t \in [0,T]} k(t)$. Hence, the first admissibility condition in \eqref{eq:affine-admissibility} gives $b_0 - (d-1) k(0) \alpha \succeq 0$.
	Multiplying by $\lambda = k(0) \geq 0$ yields the required inequality. Thus, SDE$_\lambda(\bS_d^+)$ holds, and \cite[Theorem 2.14]{abijaber_Weak_2026} yields a continuous weak solution of the projected stochastic Volterra
    equation from Theorem~\ref{thm:weak-exsitence-of-R-solutions} which takes values in $\bS_d^+$ almost surely. Therefore, this process is a continuous $\bS_d^+$-valued weak solution of the original equation \eqref{eq:main-SDE}.
\end{proof}

A fundamental tool in the analysis of Volterra-type equations is the so-called resolvent of the first kind.

\begin{definition}
	$K : \R_+ \to \R^{d \times d}$ is said to admit a resolvent of the first kind if there exists a $\R^{d \times d}$-valued measure $L$ on $[0,T]$ which is of bounded variation and satisfies $K \ast L = L \ast K = \Id_d$.
\end{definition}

The resolvent of the first kind does not always exist.
However, many kernels of practical interest, such as exponential or shifted fractional kernels, do admit one.
For a detailed treatment of resolvents and Volterra equations, we refer the reader to \cite{gripenberg_Volterra_1990}.

If the resolvent of the first kind exists for the kernel $K$, the following property holds and will play a central role in what follows.
The next lemma is a direct matrix-valued analogue of \cite[Lemma 2.6]{abijaber_Affine_2019}, and its proof follows along the same lines.

\begin{lemma}\label{lm:transform-of-F-dZ}
	Let $K \in \loc^1(\R_+, \R^{d \times d})$ and let $\widetilde{K} = \Id_d \otimes K$.
	Let $X$ be a continuous $\bS_d^+$-valued process and $Z = \int b(X) dt + \int (\sqrt{X}dW Q+Q^\top dW^\top \sqrt{X})$ a continuous semimartingale with
	$b$, and $\widetilde{K} \ast \vc(dZ)$ continuous and adapted. Assume that $\widetilde{K}$ admits a resolvent
	of the first kind $L$, i.e. $L \ast \widetilde{K} = \widetilde{K} \ast L \equiv \Id_{d^2}$. Then
	\begin{align}\label{eq:transform-of-K-dZ}
		\begin{split}
			X_t - X_0 = \int_{0}^{t} K(t-s) dZ_s \quad & \Longleftrightarrow \quad \vc(X_t - X_0) = \big(\widetilde{K} \ast \vc(dZ)\big)(t) \\
			                                           & \Longleftrightarrow \quad L \ast \vc(X - X_0) = \vc(Z).
		\end{split}
	\end{align}
	In this case, for any $F \in \loc^2(\R_+, \bS_d(\C))$ such that $\vc(F^\cplxtrans)^\cplxtrans \ast L$ is right continuous and of locally bounded variation, one has
	\begin{align}\label{eq:transform-of-F-dZ}
		\begin{split}
			\Tr(F \ast dZ) & = \vc(F^\cplxtrans)^\cplxtrans \ast \vc(dZ)                                                                                                           \\
			               & = (\vc(F^\cplxtrans)^\cplxtrans \ast L)(0)\vc(X) - (\vc(F^\cplxtrans)^\cplxtrans \ast L)\vc(X_0) + d(\vc(F^\cplxtrans)^\cplxtrans \ast L) \ast \vc(X)
		\end{split}
	\end{align}
	up to $dt \otimes P$-a.e. equivalence. If $F \ast dZ$ admits a right-continuous version, then \eqref{eq:transform-of-F-dZ} holds up to indistinguishability.
\end{lemma}

\begin{proof}
	The first equivalence in \eqref{eq:transform-of-K-dZ} follows immediately from the properties of the vectorization operator.
	Assume now that $\vc(X_t - X_0) = \big(\widetilde{K} \ast \vc(dZ)\big)(t)$, $t \in [0,T]$. Apply $L$ to both sides to get
	\begin{align*}
		L \ast \vc(X - X_0) = L \ast \big(\widetilde{K} \ast \vc(dZ)\big) = (L \ast \widetilde{K}) \ast \vc(dZ) = \Id_{d^2} \ast \vc(dZ) = \vc(Z),
	\end{align*}
	where the second equality follows from the associativity of the convolution. Conversely, assume $L \ast \vc(X - X_0) = \vc(Z)$. Then
	\begin{align*}
		\Id_{d^2} \ast \vc(X-X_0) & = (\widetilde{K} \ast L)\ast \vc(X-X_0)                \\
		                          & = \widetilde{K} \ast \big(L \ast \vc(X-X_0)\big)       \\
		                          & = \widetilde{K} \ast \vc(Z)                            \\
		                          & = \widetilde{K} \ast \big(\Id_{d^2} \ast \vc(dZ)\big)  \\
		                          & = \Id_{d^2} \ast \big(\widetilde{K} \ast \vc(dZ)\big),
	\end{align*}
	using \cite[Theorem 3.6.1(ix)]{gripenberg_Volterra_1990} for the second equality and associativity for the last equality. Since both $X-X_0$ and $\widetilde{K}\ast \vc(dZ)$ are
	continuous, they must be equal. This shows \eqref{eq:transform-of-K-dZ}. To prove \eqref{eq:transform-of-F-dZ}, observe that by the compatibility of the vectorization with inner products, we have
	\begin{equation*}
		\Tr(F \ast dZ) = \vc(F^\cplxtrans)^\cplxtrans \ast \vc(dZ).
	\end{equation*}
	The assumption that $\vc(F^\cplxtrans)^\cplxtrans \ast L$ is right-continuous and of locally bounded variation implies
	\begin{equation*}
		\vc(F^\cplxtrans)^\cplxtrans \ast L = (\vc(F^\cplxtrans)^\cplxtrans \ast L)(0) + d(\vc(F^\cplxtrans)^\cplxtrans \ast L) \ast \Id_{d^2}.
	\end{equation*}
	Convolving this with $\widetilde{K}$, using associativity of the convolution and comparing the densities of the resulting absolutely continuous functions, we deduce
	\begin{equation*}
		\vc(F^\cplxtrans)^\cplxtrans = (\vc(F^\cplxtrans)^\cplxtrans \ast L)(0)\widetilde{K} + d(\vc(F^\cplxtrans)^\cplxtrans \ast L) \ast \widetilde{K} \quad \text{a.e.}
	\end{equation*}
	By \cite[Lemma 2.1]{abijaber_Affine_2019}, convolution with respect to a semimartingale \eqref{def:ast-local-martingale} inherits the associativity property.
	Applying this associativity and using that $\vc(X_t - X_0) = \big(\widetilde{K} \ast \vc(dZ)\big)(t)$ by assumption, it follows that
	\begin{align*}
		\vc(F^\cplxtrans)^\cplxtrans \ast \vc(dZ) & = (\vc(F^\cplxtrans)^\cplxtrans \ast L)(0)\widetilde{K} \ast \vc(dZ) + d(\vc(F^\cplxtrans)^\cplxtrans \ast L) \ast (\widetilde{K} \ast \vc(dZ))       \\
		                                          & = (\vc(F^\cplxtrans)^\cplxtrans \ast L)(0)\vc(X) - (\vc(F^\cplxtrans)^\cplxtrans \ast L)\vc(X_0) + d(\vc(F^\cplxtrans)^\cplxtrans \ast L) \ast \vc(X)
	\end{align*}
	holds $dt \otimes P$-almost everywhere. The final statement is clear from right continuity of $\vc(F^\cplxtrans)^\cplxtrans \ast L$ and $d(\vc(F^\cplxtrans)^\cplxtrans \ast L) \ast X$.
\end{proof}

\section{Conditional Fourier--Laplace functional of affine Volterra process on symmetric and positive semidefinite matrices}\label{section:Fourier--Laplace}
Our main goal is to show that affine Volterra processes on $\bS_d^+$ are affine in the sense that their Fourier--Laplace transform admits a semi-closed form which is affine with respect to the past trajectory of the process,
following the results to the works of \cite{abijaber_Affine_2019} and \cite{ackermann_Inhomogeneous_2022}.

\begin{theorem}\label{thm:main-thm}
	Let $T < \infty$. Let $X= (X_t)_{t \in [0,T]}$ be an affine Volterra process on $\bS_d^+$ and fix some $u \in \bS_d(\C)$, and $f \in L^1([0,T], \bS_d(\C))$.
	Assume $\psi \in L^2([0,T], \bS_d(\C))$ solves the Riccati--Volterra equation
	\begin{equation}\label{eq:riccati_1}
		\psi(t) = uK(T-t) + \int_{t}^{T} \bigg(f(s) + B^\top (\psi(s)) - 2 \psi(s)\alpha \psi(s)\bigg) K(s-t) \, ds, \quad t \in [0,T].
	\end{equation}
	Then the process $\{Y_t : 0 \leq t \leq T\}$ defined by
	\begin{align}\label{eq:Y-t}
		\begin{split}
			Y_t & = Y_0 + \Tr\biggl( \int_{0}^{t} \psi(s) \sqrt{X_s} dW_s Q + \int_{0}^{t} Q^\top dW_s^\top\sqrt{X_s}\psi(s) + 2\int_{0}^{t} X_s\psi(s)\alpha\psi(s) \, ds\biggr) \\
			Y_0 & = \Tr\biggl(uX_0 + \int_{0}^{T} f(s)X_0 + \psi(s)b(X_0) - 2X_0\psi(s) \alpha \psi(s) \, ds \biggr)
		\end{split}
	\end{align}
	satisfies
	\begin{align}\label{eq:Y-tilde}
		Y_t & = \Tr \biggl(E\biggl[uX_T + \int_{0}^{T} f(s)X_s ds  \,\bigg|\, \cF_t\biggr] - 2 \int_{t}^{T} E\left[X_s \mid \cF_t\right] \psi(s)\alpha \psi(s) \,ds \biggr).
	\end{align}
	The process $\{\exp(-Y_t) : 0 \leq t \leq T\}$ is a local martingale and, if it is a true martingale, one has the exponential-affine transform formula
	\begin{equation}\label{eq:main-theorem-laplace-transform}
		E\left[\exp\biggl(-\Tr \Bigl(uX_T + \int_{0}^{T} f(s)X_s ds \Bigr) \biggr) \,\bigg|\, \cF_t\right] = \exp(-Y_t).
	\end{equation}
\end{theorem}

\begin{remark}
	For all \(v, x \in \bS_d(\C)\), we have the identity
	\begin{equation*}
		\sum_{i,j,k,l} v_{ij} A_{ijkl}(x) v_{kl} = 4 \Tr(x v \alpha v).
	\end{equation*}
	With this notation, since $\psi$ takes values in $\bS_d(\C)$, Equation \eqref{eq:Y-tilde} can alternatively be expressed in terms of \(A_{ijkl}\) as
	\begin{align*}
		Y_t & = E\bigg[\Tr\Big(uX(T) + \int_{0}^{T} f(s) X_s \, ds \Big) \,\bigg|\, \cF_t \bigg]
		- \frac{1}{2} \int_{t}^{T} \sum_{i,j,k,l} \psi_{ij}(s) A_{ijkl}\big(E[X_s \mid \cF_t]\big) \psi_{kl}(s) \, ds,
	\end{align*}
	which emphasizes the role of the quadratic variation, in line with the representation used in \cite{abijaber_Affine_2019}.
	A similar rewriting can be applied to equation \eqref{eq:Y-t}.
\end{remark}

\begin{proof}[Proof of Theorem \ref{thm:main-thm}]
	Let $(\widetilde{Y}_t)_{t \in [0,T]}$ be given by
	\begin{equation*}
		\widetilde{Y}_t = \Tr \biggl(E\biggl[uX_T + \int_{0}^{T} f(s)X_s ds  \,\bigg|\, \cF_t\biggr] - 2 \int_{t}^{T} E\left[X_s \mid \cF_t\right] \psi(s)\alpha \psi(s) \,ds\biggr).
	\end{equation*}

	To prove the theorem we first show that $\widetilde{Y}_0 = Y_0$.
	\begin{align}\label{eq:temp0}
		\begin{split}
			\widetilde{Y}_0 - Y_0 & = \Tr\biggl(E \left[ u(X_T-X_0) \mid \cF_0 \right] + \int_{0}^{T} f(s)E[X_s-X_0 \mid \cF_0] ds               \\
			                      & \quad - \int_{0}^{T} \psi(s)b(X_0) \, ds - 2\int_{0}^{T} E[X_s - X_0\mid \cF_0] \psi(s)\alpha \psi(s) \,ds\biggr).
		\end{split}
	\end{align}
	By Lemma \ref{lm:cond-expec} we now see that the last term inside the trace has the representation
	\begin{align*}
		\int_{0}^{T} E[X_s-X_0\mid \cF_0] \psi(s)\alpha \psi(s) \,ds & = \int_{0}^{T} \left(\int_{0}^{s} K(s-r)b(E[X_r \mid \cF_0]) dr\right) \psi(s)\alpha \psi(s)\, ds \\
		\intertext{and an application of Fubini's theorem and by using the cyclicity of the trace, we get}
		                                                             & = \int_{0}^{T} b(E[X_r \mid \cF_0]) \int_{r}^{T} \psi(s)\alpha \psi(s)K(s-r) \, ds \, dr.
	\end{align*}
	Therefore, using the Riccati--Volterra equation \eqref{eq:riccati_1} it holds
	\begin{align}\label{eq:temp1}
		\begin{split}
			-2\int_{0}^{T} E[X_s-X_0\mid \cF_0] \psi(s)\alpha \psi(s) \,ds & = -\int_{0}^{T} b(E[X_r \mid \cF_0]) \biggl(\int_{r}^{T} 2\psi(s)\alpha \psi(s)K(s-r) \, ds\biggr) \, dr \\
			                                                               & = \int_{0}^{T} b(E[X_r \mid \cF_0]) \bigg(\psi(r) - uK(T-r)                                              \\
			                                                               & \quad - \int_{r}^{T} \bigl(f(s) + B^\top(\psi(s)) \bigr)K(s-r) \,ds\bigg) \, dr.
		\end{split}
	\end{align}
	Remark now that
	\begin{equation}\label{eq:temp2}
		\int_{0}^{T} \psi(s)b(X_0) \, ds = \int_{0}^{T} \psi(s)b(E[X_s \mid \cF_0])\, ds - \int_{0}^{T} \psi(s)B(E[X_s-X_0 \mid \cF_0])\, ds
	\end{equation}
	and by Lemma \ref{lm:cond-expec}
	\begin{equation}\label{eq:temp3}
		u E \left[(X_T-X_0) \mid \cF_0 \right] = u\int_{0}^{T} K(T-s)b(E[X_s \mid \cF_0]) \, ds.
	\end{equation}
	Substituting \eqref{eq:temp1}, \eqref{eq:temp2} and \eqref{eq:temp3} into \eqref{eq:temp0} and using cyclicity and Lemma \ref{lm:cond-expec},
	\begin{align*}
		\widetilde{Y}_0 - Y_0 & = \Tr\biggl( u \int_{0}^{T} K(T-s)b(E[X_s \mid \cF_0]) \, ds + \int_{0}^{T} f(s)E[X_s-X_0 \mid \cF_0] \, ds         \\
		                      & \quad - \int_{0}^{T} \psi(s)b(E[X_s \mid \cF_0])\, ds + \int_{0}^{T} \psi(s)B(E[X_s-X_0 \mid \cF_0])\, ds           \\
		                      & \quad + \int_{0}^{T} b(E[X_r \mid \cF_0]) \bigg(\psi(r) - uK(T-r)                                                   \\
		                      & \quad - \int_{r}^{T} \bigl(f(s) + B^\top(\psi(s)) \bigr)K(s-r) \,ds\bigg) \, dr \biggr)                             \\
		                      & = \Tr\biggl( \int_{0}^{T} f(s)E[X_s-X_0 \mid \cF_0] + \psi(s)B(E[X_s-X_0 \mid \cF_0]) \, ds                         \\
		                      & \quad - \int_{0}^{T} b(E[X_r \mid \cF_0]) \int_{r}^{T} \bigl(f(s) + B^\top(\psi(s))\bigr)K(s-r) \,ds \, dr \biggr).
	\end{align*}
	It remains to show, under the trace, that
	\begin{align*}
		\int_{0}^{T} f(s)E[X_s-X_0 \mid \cF_0] \, ds       & = \int_{0}^{T} b(E[X_r \mid \cF_0]) \int_{r}^{T} f(s)K(s-r) \,ds \, dr              \\
		\int_{0}^{T} \psi(s)B(E[X_s-X_0 \mid \cF_0]) \, ds & = \int_{0}^{T} b(E[X_r \mid \cF_0]) \int_{r}^{T} B^\top(\psi(s)) K(s-r) \,ds \, dr.
	\end{align*}
	The first equation immediately follows by Lemma \ref{lm:cond-expec}, Fubini's theorem and cyclicity
	\begin{align*}
		\int_{0}^{T} f(s)E[X_s-X_0 \mid \cF_0] \, ds & = \int_{0}^{T} f(s) \left(\int_{0}^{s} K(s-r)b(E[X_r \mid \cF_0])\, dr \right) ds    \\
		                                             & = \int_{0}^{T} b(E[X_r \mid \cF_0]) \left(\int_{r}^{T} f(s) K(s-r) \, ds \right) dr.
	\end{align*}
	And similarly for the second equation, by making use of the linearity of $B$ and the definition of its adjoint operator $B^\top$, we obtain
	\begin{align*}
		\int_{0}^{T} \psi(s)B(E[X_s-X_0 \mid \cF_0]) \, ds & = \int_{0}^{T} \psi(s)B\left(\int_{0}^{s} K(s-r)b(E[X_r \mid \cF_0])\, dr\right) \, ds \\
		                                                   & = \int_{0}^{T} \int_{0}^{s} \psi(s)B\Big(K(s-r)b(E[X_r \mid \cF_0])\Big) \, dr \, ds   \\
		                                                   & = \int_{0}^{T} b(E[X_r \mid \cF_0]) \int_{r}^{T} B^\top (\psi(s))K(s-r) \, ds \, dr.
	\end{align*}
	We have thus shown that $\widetilde{Y}_0 = Y_0$.

	Next, in order to prove that $\widetilde{Y} = Y$, let $t \in [0,T]$ and observe that
	\begin{align*}
		\widetilde{Y}_t & = \Tr \biggl(E\biggl[uX_T + \int_{0}^{T} f(s)X_s ds \,\bigg|\, \cF_t\biggr] - 2 \int_{0}^{T} E\left[X_s \mid \cF_t\right] \psi(s)\alpha \psi(s) \,ds \\
		                & \quad + 2 \int_{0}^{t} E\left[X_s \mid \cF_t\right] \psi(s)\alpha \psi(s) \,ds\biggr)                                                                \\
		                & = \Tr \biggl( E[ uX_T \mid \cF_t ] + \int_{0}^{T} f(s) E[X_s \mid \cF_t ] ds - 2\int_{0}^{T} E[X_s \mid \cF_t ] \psi(s) \alpha \psi(s) \, ds         \\
		                & \quad + 2\int_{0}^{t} X_s \psi(s) \alpha \psi(s) \, ds\biggr).
	\end{align*}
	Let $C \in \C$ be a constant independent of $t$ and note that we allow $C$ to change from line to line. Then
	\begin{align*}
		\widetilde{Y}_t - Y_t & = C + \Tr \biggl( uE[X_T \mid \cF_t ] + \int_{0}^{T} f(s) E[X_s \mid \cF_t ] ds - 2\int_{0}^{T} E[X_s \mid \cF_t ] \psi(s) \alpha \psi(s) \, ds \\
		                      & \quad - \int_{0}^{t} \psi(s) \sqrt{X_s}dW_s Q - \int_{0}^{t} Q^\top dW_s^\top \sqrt{X_s}\psi(s)\biggr).                                         \\
	\end{align*}
	Applying Lemma \ref{lm:cond-expec} and cyclicity we see that
	\begin{align*}
		\widetilde{Y}_t - Y_t & = C + \Tr \biggl(\int_{0}^{T} f(s) E[X_s \mid \cF_t ] \, ds - 2\int_{0}^{T} E[X_s \mid \cF_t ] \psi(s) \alpha \psi(s) \, ds \\
		                      & \quad + u \int_{0}^{T} K(T-s) b(E[X_s \mid \cF_t]) \, ds  + u\int_{0}^{t} K(T-s)\sqrt{X_s}dW_s Q                            \\
		                      & \quad + u\int_{0}^{t} K(T-s)Q^\top dW_s^\top \sqrt{X_s}
		- \int_{0}^{t} \psi(s) \sqrt{X_s}dW_s Q - \int_{0}^{t} Q^\top dW_s^\top \sqrt{X_s}\psi(s)\biggr)                                                    \\
		                      & = C + \Tr \biggl(\int_{0}^{T} f(s) E[X_s \mid \cF_t ] \, ds - 2\int_{0}^{T} E[X_s \mid \cF_t ] \psi(s) \alpha \psi(s) \, ds \\
		                      & \quad  + u \int_{0}^{T} K(T-s) b(E[X_s \mid \cF_t]) \, ds  + \int_{0}^{t} \left(uK(T-s)-\psi(s)\right)\sqrt{X_s}dW_s Q      \\
		                      & \quad + \int_{0}^{t} Q^\top dW_s^\top \sqrt{X_s}\left( uK(T-s)-\psi(s)\right)\biggr).
	\end{align*}
	Furthermore, using $b(\cdot) = b_0 + B(\cdot)$ and rewriting the third term inside the trace with the adjoint operator $B^\top$, we see that
	\begin{align*}
		\widetilde{Y}_t - Y_t & = C + \Tr \biggl(\int_{0}^{T} f(s) E[X_s \mid \cF_t ] \, ds + \int_{0}^{T} E[X_s \mid \cF_t] \Bigl(B^\top\bigl( u K(T-s)\bigr) - 2\psi(s) \alpha \psi(s)\Bigr) ds \\
		                      & \quad + \int_{0}^{t} \Bigl( uK(T-s)-\psi(s)\Bigr)\sqrt{X_s}dW_s Q + \int_{0}^{t} Q^\top dW_s^\top \sqrt{X_s}\Bigl(uK(T-s)-\psi(s)\Bigr)\biggr).
	\end{align*}
	Applying now Lemma \ref{lm:cond-expec} once again, we obtain
	\begin{align}\label{eq:temp-1}
		\begin{split}
			\widetilde{Y}_t & - Y_t = C + \Tr \biggl(\int_{0}^{T} f(s) E[X_s \mid \cF_t ] \, ds                                                                                                 \\
			                & \quad + \int_{0}^{T} \left(\int_{0}^{s} K(s-r)b(E[X_r \mid \cF_t]) dr\right) \Bigl(B^\top\bigl( u K(T-s)\bigr) - 2\psi(s) \alpha \psi(s)\Bigr) ds                 \\
			                & \quad + \int_{0}^{T} \left(\int_{0}^{t \wedge s} K(s-r)\sqrt{X_r}dW_r Q\right) \Bigl(B^\top\bigl( u K(T-s)\bigr) - 2\psi(s) \alpha \psi(s)\Bigr) ds \\
			                & \quad + \int_{0}^{T} \left(\int_{0}^{t \wedge s} K(s-r)Q^\top dW_r^\top \sqrt{X_r}\right) \Bigl(B^\top\bigl( u K(T-s)\bigr) - 2\psi(s) \alpha \psi(s)\Bigr) ds \\
			                & \quad + \int_{0}^{t} \Bigl( uK(T-s)-\psi(s)\Bigr)\sqrt{X_s}dW_s Q                                                                                                 \\
			                & \quad + \int_{0}^{t} Q^\top dW_s^\top \sqrt{X_s}\Bigl( uK(T-s)-\psi(s)\Bigr)\biggr).
		\end{split}
	\end{align}
	Note that by the Riccati--Volterra equation \eqref{eq:riccati_1} it holds for all $s \in [0,t]$
	\begin{equation*}
		uK(T-s)-\psi(s) = - \int_{s}^{T} \bigl(f(r) + B^\top(\psi(r)) - 2\psi(r)\alpha \psi(r)\bigr)K(r-s) \, dr.
	\end{equation*}

	The stochastic Fubini theorem (see \cite[Theorem 2.2]{veraar_stochastic_2012}) implies that, again understood under the trace operator,
	\begin{align}\label{eq:temp-2}
		\begin{split}
			 & \int_{0}^{t} \Bigl( uK(T-s)-\psi(s)\Bigr)\sqrt{X_s}dW_s Q                                                                                         \\
			 & \quad = - \int_{0}^{t} \biggl( \int_{s}^{T} \bigl(f(r) + B^\top(\psi(r)) - 2\psi(r)\alpha \psi(r)\bigr)K(r-s) \, dr \biggr) \sqrt{X_s}dW_s Q      \\
			 & \quad = - \int_{0}^{T} \biggl(\int_{0}^{t \wedge r} K(r-s)\sqrt{X_s}dW_s Q \biggr) \Bigl(f(r) + B^\top(\psi(r)) - 2\psi(r)\alpha \psi(r)\Bigr) dr
		\end{split}
	\end{align}
	and similarly
	\begin{align}\label{eq:temp-3}
		\begin{split}
			 & \int_{0}^{t} Q^\top dW_s^\top \sqrt{X_s}\Bigl( uK(T-s)-\psi(s)\Bigr)                                                                                          \\
			 & \quad = - \int_{0}^{t} \biggl( \int_{s}^{T} \bigl(f(r) + B^\top(\psi(r)) - 2\psi(r)\alpha \psi(r)\bigr)K(r-s) \, dr\biggr) Q^\top dW_s^\top \sqrt{X_s}        \\
			 & \quad = - \int_{0}^{T} \biggl(\int_{0}^{t \wedge r} K(r-s)Q^\top dW_s^\top \sqrt{X_s} \biggr) \Bigl(f(r) + B^\top(\psi(r)) - 2\psi(r)\alpha \psi(r)\Bigr) dr.
		\end{split}
	\end{align}

	To justify the application of the stochastic Fubini theorem, we will work with the vectorized process of \eqref{eq:temp-2}. To this end, note first, that as before
	\begin{align*}
		\int_{0}^{t} & \biggl( \int_{s}^{T} \bigl(f(r) + B^\top(\psi(r)) - 2\psi(r)\alpha \psi(r)\bigr)K(r-s) \, dr\biggr) \sqrt{X_s}dW_s Q                                                   \\
		             & \quad = \int_{0}^{t} \biggl( \int_{s}^{T} \bigl(f(r) + B^\top(\psi(r)) - 2\psi(r)\alpha \psi(r)\bigr)K(r-s) \, dr\biggr) \sum_{k,l = 1}^{d} \sigma^{kl}(X_s) dW_{s,kl}
	\end{align*}
	with $\sigma^{kl}(x) = \sqrt{x}e_{kl}Q$ for $x \in \bS_d^+$, where $e_{kl}$ is the matrix with a $1$ at position $(k,l)$ and zeros elsewhere. Now we see that
	\begin{align*}
		\Tr & \biggl( \int_{0}^{t} \biggl( \int_{s}^{T} \bigl(f(r) + B^\top(\psi(r)) - 2\psi(r)\alpha \psi(r)\bigr)K(r-s) \, dr \biggr) \sum_{k,l = 1}^{d} \sigma^{kl}(X_s) dW_{s,kl} \biggr)                                                  \\
		    & \quad = \int_{0}^{t} \vc\biggl(\Bigl(\int_{s}^{T} \bigl(f(r) + B^\top(\psi(r)) - 2\psi(r)\alpha \psi(r)\bigr)K(r-s) \, dr \Bigr)^\cplxtrans \biggr)^\cplxtrans \sum_{k,l = 1}^{d} \vc(\sigma^{kl}(X_s)) dW_{s,kl}             \\
		    & \quad = \int_{0}^{t} \int_{s}^{T} \Bigl( \bigl(K(r-s)^\top \otimes \Id_d\bigr) \vc\bigl(f(r) + B^\top(\psi(r)) - 2\psi(r)\alpha \psi(r)\bigr)^\cplxtrans \Bigr)^\cplxtrans dr \sum_{k,l = 1}^{d} \vc(\sigma^{kl}(X_s)) dW_{s,kl} \\
		    & \quad = \sum_{k,l = 1}^{d} \int_{0}^{t} \int_{s}^{T} \vc\Bigl(\bigl(f(r) + B^\top \bigl(\psi(r)\bigr) - 2\psi(r)\alpha \psi(r)\bigr)^\cplxtrans \Bigr)^\cplxtrans \widetilde{K}(r-s) \vc(\sigma^{kl}(X_s))\, dr\,  dW_{s,kl},
	\end{align*}
	with $\widetilde{K} = \bigl(\Id_d \otimes K\bigr)$. This now matches the expression of the justification of the stochastic Fubini theorem in \cite{ackermann_Inhomogeneous_2022}, and
	we can simply reproduce the argument. Let for every $k,l \in \{1, \dotsc, d\}$, $\varphi_{kl} : [0,T] \times [0,t] \times \bS_d^+ \to \C$, $\varphi_{kl}(r,s,x) = \vc\Bigl(\bigl(f(r) + B^\top(\psi(r)) - 2\psi(r)\alpha \psi(r)\bigr)^\cplxtrans \Bigr)^\cplxtrans \widetilde{K}(r-s) \vc(\sigma^{kl}(x)) \Ind_{[0,r]}(s)$.
	It now holds for every $k,l \in \{1, \dotsc, d\}$
	\begin{align}\label{eq:temp-4}
		\begin{split}
			\int_{0}^{T} \biggl( \int_{0}^{t} \| \varphi_{kl}(r,s,X_s) \|^2 ds \biggr)^{\frac{1}{2}} dr &\leq \int_{0}^{T} \| f(r) + B^\top \bigl(\psi(r)\bigr) - 2\psi(r)\alpha \psi(r) \| \biggl(\int_{0}^{t \wedge r} \|\widetilde{K}(r-s)\|^2 \|\sigma^{kl}(X_s)\|^2 ds\biggr)^{\frac{1}{2}}dr \\
			                                                                                   & \leq \biggl(\sup_{s \in [0,T]} \|\sigma^{kl}(X_s)\| \biggr) \biggl(\int_{0}^{T} \|\widetilde{K}(s)\|^2 ds\biggr)^\frac{1}{2}                                                           \\
			                                                                                   & \quad \times \int_{0}^{T} \Bigl( \|f(r)\| + \|B^\top(\psi(r))\| + \|2\psi(r)\alpha \psi(r)\|\Bigr) dr.
		\end{split}
	\end{align}
	Since $\sigma^{kl}$ is at most of linear growth and $X$ is continuous we obtain that $\sup_{s \in [0,T]} \|\sigma^{kl}(X_s)\| < \infty$. Moreover, by construction, there exists $c \in (0, \infty)$ such that for all $r \in [0,T]$, $\|B^\top(\psi(r))\| \leq c \|\psi(r)\|$ and $\|\psi(r)\alpha \psi(r)\| \leq c \|\psi(r)\|^2$.
	Now, we conclude from $K \in \loc^2(\R_+, \R^{d \times d})$, $f \in L^1([0,T], \bS_d(\C))$ and $\psi \in L^2([0,T], \C^{d \times d})$ that \eqref{eq:temp-4} is finite. We can thus apply \cite[Theorem 2.2]{veraar_stochastic_2012} to each $\R$-valued integral
	$\int_{0}^{t} \int_{0}^{T} \re \varphi_{kl}(r,s,X_s) dr dW_{s,kl}$, $\int_{0}^{t} \int_{0}^{T} \im \varphi_{kl}(r,s,X_s) dr dW_{s,kl}$.
	The same reasoning applies for \eqref{eq:temp-3} with $\sigma^{kl}(x) = Q^\top e_{lk} \sqrt{x}$.

	Combining \eqref{eq:temp-1}, \eqref{eq:temp-2} and \eqref{eq:temp-3} implies that
	\begin{align*}
		\widetilde{Y}_t - Y_t & = C + \Tr \biggl(\int_{0}^{T} f(s) E[X_s \mid \cF_t ] \, ds                                                                                                             \\
		                      & \quad + \int_{0}^{T} \biggl(\int_{0}^{s} K(s-r)b(E[X_r \mid \cF_t]) dr\biggr) \Bigl(B^\top\bigl(u K(T-s)\bigr) - 2\psi(s) \alpha \psi(s)\Bigr) ds                       \\
		                      & \quad + \int_{0}^{T} \biggl(\int_{0}^{t \wedge s} K(s-r)\sqrt{X_r}dW_r Q\biggr) \Bigl( B^\top\bigl(u K(T-s)\bigr) - f(s) - B^\top(\psi(s))\Bigr) ds                     \\
		                      & \quad + \int_{0}^{T} \biggl(\int_{0}^{t \wedge s} K(s-r)Q^\top dW_r^\top \sqrt{X_r} \biggr) \Bigl(B^\top\bigl( u K(T-s)\bigr) - f(s) - B^\top(\psi(s))\Bigr) ds \biggr) \\
		                      & = C + \Tr \biggl(\int_{0}^{T} f(s) E[X_s \mid \cF_t ] \, ds                                                                                                             \\
		                      & \quad + \int_{0}^{T} \biggl(\int_{0}^{s} K(s-r)b(E[X_r \mid \cF_t]) dr\biggr) \Bigl(B^\top\bigl(u K(T-s)\bigr) - 2\psi(s) \alpha \psi(s)\Bigr) ds                       \\
		                      & \quad + \int_{0}^{T} \biggl(\int_{0}^{t \wedge s} K(s-r)\sqrt{X_r}dW_r Q + \int_{0}^{t \wedge s} K(s-r)Q^\top dW_r^\top \sqrt{X_r}\biggr)                               \\
		                      & \quad \times \Bigl(B^\top\bigl( u K(T-s) - \psi(s)\bigr) - f(s)\Bigr) ds \biggr)
	\end{align*}
	Observe that by Lemma \ref{lm:cond-expec} it holds for all $s \in [0,T]$
	\begin{align*}
		\int_{0}^{t \wedge s} K(s-r) & \sqrt{X_r}dW_r Q + \int_{0}^{t \wedge s} K(s-r)Q^\top (dW_r)^\top \sqrt{X_r} \\
		                             & = E[X_s \mid \cF_t] - X_0 - \int_{0}^{s} K(s-r) b(E[X_r \mid \cF_t]) dr
	\end{align*}
	which implies
	\begin{align*}
		\widetilde{Y}_t & - Y_t = C + \Tr \biggl(\int_{0}^{T} f(s) E[X_s \mid \cF_t ] \, ds                                                                                     \\
		                & \quad + \int_{0}^{T} \biggl(\int_{0}^{s} K(s-r)b(E[X_r \mid \cF_t]) dr\biggr) \Bigl(B^\top\bigl(u K(T-s)\bigr) - 2\psi(s) \alpha \psi(s)\Bigr) ds     \\
		                & \quad + \int_{0}^{T} \biggl(E[X_s \mid \cF_t] - X_0 - \int_{0}^{s} K(s-r) b(E[X_r \mid \cF_t]) dr\biggr)                                              \\
		                & \quad \times \Bigl(B^\top\bigl( u K(T-s)- \psi(s)\bigr) - f(s)\Bigr) ds \biggr)                                                                       \\
		                & = C + \Tr \biggl(\int_{0}^{T} E[X_s \mid \cF_t ] f(s) \, ds                                                                                           \\
		                & \quad + \int_{0}^{T} \biggl(\int_{0}^{s} K(s-r)b(E[X_r \mid \cF_t]) dr\biggr) \Bigl( f(s) + B^\top(\psi(s))-2\psi(s) \alpha \psi(s)\Bigr) ds               \\
		                & \quad + \int_{0}^{T} E[X_s \mid \cF_t]\Bigl( B^\top\bigl(u K(T-s)- \psi(s)\bigr) - f(s)\Bigr) ds \biggr)                                              \\
		                & = C + \Tr \biggl( \int_{0}^{T} \biggl(\int_{0}^{s} K(s-r)b(E[X_r \mid \cF_t]) dr\biggr) \Bigl(f(s) + B^\top(\psi(s))-2\psi(s) \alpha \psi(s)\Bigr) ds \\
		                & \quad + \int_{0}^{T} E[X_s \mid \cF_t]B^\top\bigl( u K(T-s)- \psi(s)\bigr) ds \biggr).
	\end{align*}
	We therefore conclude that
	\begin{align}
		\begin{split}\label{eq:temp**}
			\widetilde{Y}_t - Y_t & = C + \Tr \biggl( \int_{0}^{T} \biggl(\int_{0}^{s} K(s-r)B(E[X_r \mid \cF_t]) dr\biggr) \Bigl(f(s) + B^\top(\psi(s))-2\psi(s) \alpha \psi(s)\Bigr) ds \\
			                      & \quad + \int_{0}^{T} E[X_s \mid \cF_t] B^\top\bigl(u K(T-s)- \psi(s)\bigr) ds \biggr).
		\end{split}
	\end{align}
	Finally, under the trace operator, using Fubini's theorem and subsequently the Riccati--Volterra equation \eqref{eq:riccati_1}, it holds
	\begin{align}
		\begin{split}\label{eq:temp-5}
			 & \int_{0}^{T} \biggl(\int_{0}^{s} K(s-r)B(E[X_r \mid \cF_t])dr\biggr) \Bigl( f(s) + B^\top(\psi(s))-2\psi(s) \alpha \psi(s)\Bigr) ds       \\
			 & \quad = \int_{0}^{T} B(E[X_r \mid \cF_t])\biggl(\int_{r}^{T} \Bigl(f(s) + B^\top(\psi(s))-2\psi(s) \alpha \psi(s)\Bigr)K(s-r) ds\biggr)dr \\
			 & \quad = \int_{0}^{T} B\bigl(E[X_r \mid \cF_t]\bigr)\Bigl(\psi(r) - u K(T-r)\Bigr)dr                                                       \\
			 & \quad = \int_{0}^{T} E[X_r \mid \cF_t]B^\top \bigl(\psi(r) - u K(T-r)\bigr)dr.
		\end{split}
	\end{align}
	It now follows from \eqref{eq:temp**} and \eqref{eq:temp-5} that $\widetilde{Y}_t - Y_t = C$. The fact that $\widetilde{Y}_0 = Y_0$ implies that $C = 0$ and hence
	$\widetilde{Y} = Y$. For the remaining claims, note that by \eqref{eq:Y-t}, $-Y + \frac{1}{2}\langle Y \rangle$ is a local martingale and hence $\exp(-Y)$ is a local martingale.
	Since $-Y_T = -\Tr \bigl(uX_T + \int_{0}^{T} f(s)X_s \, ds\bigr)$ (cf. \eqref{eq:Y-tilde}), we have
	\begin{equation*}
		E\left[\exp\biggl(-\Tr \Bigl( uX_T + \int_{0}^{T} f(s)X_s \, ds \Bigr) \biggr) \,\bigg|\, \cF_t \right] = \exp(-Y_t)
	\end{equation*}
	for all $t \in [0,T]$ if $\exp(-Y)$ is a true martingale.

\end{proof}

In the special case where $K \equiv \Id_d$, the affine Volterra process $X$ of \eqref{eq:main-SDE} reduces to an affine process on $\bS_d^+$ as studied in \cite{cuchiero_Affine_2011}.
Note also in this case that \eqref{eq:riccati_1} reduces to the classical Riccati equation for affine processes on $\bS_d^+$.
It is known (see e.g. \cite[Theorem 2.4]{cuchiero_Affine_2011}) that the conditional Fourier--Laplace transform of such an affine process admits an exponential-affine representation
\begin{equation*}
	E\left[\exp\left( -\Tr(uX_T) \right) \mid \cF_t \right] = \exp\Bigl(-\phi(T-t)-\Tr(\chi(T-t)X_t)\Bigr), \quad t \in [0,T],
\end{equation*}
for $u \in \bS_d^+$, where
$\phi$ and $\chi$ are given in terms of the solution to the corresponding Riccati equation.
Such an exponential structure holds also true in the Volterra case at time zero, as expressed in the following remark.

\begin{remark}\label{remark:time-zero-classical}
	Under the notation of Theorem \ref{thm:main-thm} it holds for $\phi : [0,T] \to \C$ and $\chi :[0,T] \to \bS_d(\C)$ defined by
	\begin{align}\label{eq:riccati-at-time-zero}
		\begin{split}
			\phi(t) & = \int_{T-t}^{T} \Tr(\psi(s)b_0) \, ds, \quad t \in [0,T],                                               \\
			\chi(t) & = u + \int_{T-t}^{T} f(s) + B^\top \bigl(\psi(s)\bigr) - 2\psi(s)\alpha \psi(s)\, ds, \quad t \in [0,T],
		\end{split}
	\end{align}
	that
	\begin{equation*}
		Y_0 = \phi(T) + \Tr(\chi(T)X_0).
	\end{equation*}
	It follows that if $e^{-Y}$ is a true martingale, then
	\begin{equation}\label{eq:laplace-transform-at-time-zero}
		E\left[\exp\biggl(-\Tr \Bigl( uX_T + \int_{0}^{T} f(s)X_s \, ds \Bigr) \biggr)\right] = \exp\Bigl(-\phi(T) - \Tr(\chi(T)X_0)\Bigr).
	\end{equation}
\end{remark}

The following proposition extends the transform formula of Theorem~\ref{thm:main-thm} by allowing for an additional Brownian functional driven by the same Brownian motion as the affine Volterra process. More precisely, we consider a functional of the form
$$ N_t = \int_{0}^{t} \Tr\bigl(H(s)^\top \sqrt{X_s} \, dW_s\bigr),$$
for a deterministic matrix-valued function $H$. 
Such Brownian functionals and the resulting cross-variation terms are familiar from affine stochastic volatility models, in particular from the Volterra--Heston setting and matrix-valued affine covariance models.
The relevant point is that the additional quadratic-variation terms generated by $N$ preserve the affine structure. Indeed, both its quadratic variation and its cross-variation with the martingale part of $X$ are linear in the state. Consequently, the additional finite-variation terms arising from It\^o's formula are of the form $\Tr(X_t A(t))dt$ and can therefore be absorbed into the Riccati--Volterra equation. The exponential-affine structure of the transform is thus preserved.

This extension will be particularly useful in the applications considered in Section~\ref{section:Commodity-models}, where the quantities of interest naturally involve Brownian functionals that are correlated with the Brownian motion driving the Volterra covariance process.

\begin{proposition}\label{prop:additional-brownian-functional}
	Let $T < \infty$ and let $X = (X_t)_{t \in [0,T]}$ be an affine Volterra
	process on $\bS_d^+$, driven by
	the $d \times d$ Brownian motion $W$. Let $H \in L^2([0,T],\R^{d \times d})$ be deterministic, let $f \in L^1([0,T],\bS_d(\C))$, $z \in \C$, $u \in \bS_d(\C)$ and define the real-valued continuous martingale
	\begin{equation*}
		N_t := \int_{0}^{t} \Tr\bigl( H(s)^\top\sqrt{X_s}\,dW_s \bigr), \quad t \in [0,T].
	\end{equation*}
	Assume that $\psi\in L^2([0,T],\bS_d(\C))$ solves
	\begin{equation}\label{eq:riccati-additional-brownian}
		\psi(t) = uK(T-t) + \int_{t}^{T} \Bigl( f(s) + B^\top(\psi(s)) + z\cC_{H(s)}(\psi(s)) - \frac{z^2}{2}H(s)H(s)^\top - 2\psi(s)\alpha\psi(s) \Bigr) K(s-t)\,ds,
	\end{equation}
	where $\cC_H(\psi) := HQ\psi+\psi Q^\top H^\top$.
	Define
	\begin{align}\label{eq:Y-additional-brownian}
		\begin{split}
			Y_t^{z,H} = Y_0^{z,H} &+ \Tr \bigg( \int_{0}^{t} \psi(s)\sqrt{X_s}\,dW_sQ + \int_{0}^{t} Q^\top dW_s^\top\sqrt{X_s}\psi(s)\\
			& \quad + \int_{0}^{t} X_s
		\left[ 2\psi(s)\alpha\psi(s) +\frac{z^2}{2}H(s)H(s)^\top -z \cC_{H(s)}(\psi(s))
				\right]ds \bigg)
		\end{split}
	\end{align}
	with
	\begin{align}\label{eq:Y0-additional-brownian}
		\begin{split}
			Y_0^{z,H} &= \Tr\bigg( uX_0 + \int_{0}^{T} \Big[ f(s)X_0 + \psi(s)b(X_0) + z X_0\cC_{H(s)}(\psi(s)) \\
			& \quad -\frac{z^2}{2}X_0H(s)H(s)^\top - 2X_0\psi(s)\alpha\psi(s) \Big] ds \bigg).
		\end{split}
	\end{align}
	The process $\{\exp(zN_t - Y_t^{z,H}) : 0 \leq t \leq T\}$ is a local martingale and, if it is a true martingale, one has the exponential-affine transform formula
	\begin{equation}\label{eq:transform-additional-brownian}
		E\left[\exp\biggl( zN_T - \Tr \Bigl(uX_T + \int_{0}^{T} f(s)X_s ds \Bigr) \biggr) \,\bigg|\, \cF_t\right] = \exp(zN_t - Y_t^{z,H}).
	\end{equation}
\end{proposition}

\begin{proof}
	The argument follows the proof of Theorem~\ref{thm:main-thm}. We only indicate the additional
	terms generated by the process $N$.

	Note first that, since $X$ has continuous paths,
	\begin{equation*}
		\int_{0}^{T} \|\sqrt{X_s}H(s)\|^2 \, ds = \int_{0}^{T} \Tr \bigl(X_s H(s)H(s)^\top \bigr) ds \leq \sup_{s \in [0,T]} \|X_s\| \int_{0}^{T} \|H(s)\|^2 \, ds < \infty
	\end{equation*}
	almost surely. Moreover, Lemma~\ref{lem:moment_bound} gives
	\begin{equation*}
		E\biggl[\int_{0}^{T} \Tr \bigl(X_s H(s)H(s)^\top\bigr) \, ds\biggr] \leq \sup_{s \in [0,T]} E[\|X_s\|] \int_{0}^{T} \|H(s)\|^2 \, ds < \infty.
	\end{equation*}
	Thus, $N$ is well-defined and is a square-integrable continuous martingale.
	Let $M^X$ denote the martingale part of $X$, that is,
	\begin{equation*}
		dM_t^X = \sqrt{X_t}\,dW_tQ + Q^\top dW_t^\top\sqrt{X_t}.
	\end{equation*}
	The quadratic variation of $N$ is
	\begin{equation}\label{eq:qv-N}
		d\langle N\rangle_t = \Tr\bigl(H(t)^\top X_t H(t) \bigr) dt = \Tr\bigl(X_t H(t)H(t)^\top \bigr) dt,
	\end{equation}
	while its matrix-valued cross-variation with $M^X$ is
	\begin{equation*}
		d\langle N, M^X \rangle_t = \bigl( X_t H(t)Q + Q^\top H(t)^\top X_t \bigr) dt.
	\end{equation*}
	Consequently,
	\begin{equation}\label{eq:cross-N-Y}
		d\biggl\langle N, \Tr\biggl(\int_{0}^{\cdot} \psi(s) dM_s^X\biggr)\biggr\rangle_t = \Tr \bigl(\psi(t) d\langle N, M^X\rangle_t \bigr) = \Tr \bigl( X_t \cC_{H(t)}(\psi(t)) \bigr) dt.
	\end{equation}
	Set
	\begin{equation*}
		L_t := zN_t - Y_t^{z,H}.
	\end{equation*}
	The continuous local martingale part of $L$ is given by
	\begin{equation*}
		dL_t^m = zdN_t - \Tr \bigl( \psi(t)\sqrt{X_t}dW_t Q + Q^\top dW_t^\top \sqrt{X_t} \psi(t) \bigr).
	\end{equation*}
	Using \eqref{eq:qv-N}, \eqref{eq:cross-N-Y}, and the quadratic
	variation calculation from Theorem~\ref{thm:main-thm}, we obtain
	\begin{equation*}
		\frac{1}{2} d \langle L^m \rangle_t = \Tr \biggl(X_t \biggl[2 \psi(t) \alpha \psi(t) + \frac{z^2}{2}H(t)H(t)^\top - z\cC_{H(t)}(\psi(t))\biggr]\biggr) dt.
	\end{equation*}
	Defining now
	\begin{equation*}
		A(s) := 2 \psi(s)\alpha\psi(s) + \frac{z^2}{2}H(s)H(s)^\top - z\cC_{H(s)}(\psi(s)),
	\end{equation*}
	we have by definition of $Y^{z,H}$ 
	\begin{equation*}
		dY_t^{z,H} = \Tr \bigl(\psi(t) dM_t^X\bigr) + \Tr \bigl(X_t A(t)\bigr) dt.
	\end{equation*}
	Consequently,
	\begin{equation*}
		dL_t = dL_t^m - \frac{1}{2} d\langle L^m\rangle_t.
	\end{equation*}
	It\^o's formula therefore gives
	\begin{equation*}
		d\exp(L_t) = \exp(L_t)\,dL_t^m,
	\end{equation*}
	so that $\exp(L)$ is a local martingale. It remains to identify the terminal value of $Y^{z,H}$.
	Combining \eqref{eq:Y-additional-brownian} and \eqref{eq:Y0-additional-brownian} yields
	\begin{align*}
		Y_T^{z,H} &= \Tr(uX_0) + \int_{0}^{T} \Tr(f(s)X_0) \, ds + \int_{0}^{T} \Tr(\psi(s)b(X_0))\,ds\\
		& \quad + \Tr\biggl( \int_{0}^{T} \psi(s)\,dM_s^X\biggr) + \int_{0}^{T} \Tr\bigl(A(s)(X_s-X_0)\bigr)ds.
	\end{align*}
On the other hand, \eqref{eq:riccati-additional-brownian} can be written as
\begin{equation*}
	\psi(t) = uK(T-t) + \int_{t}^{T} \bigl(f(s)+B^\top(\psi(s))-A(s)\bigr)K(s-t)\, ds.
\end{equation*}
Applying the Fubini identities as in the proof of Theorem~\ref{thm:main-thm} to the Volterra equation for $X-X_0$, and using $b(X_s) = b(X_0) + B(X_s-X_0)$, therefore gives
\begin{align*}
	&\Tr\bigl(u(X_T-X_0)\bigr) + \int_{0}^{T} \Tr\bigl(f(s)(X_s-X_0)\bigr)ds \\
	&\quad = \int_{0}^{T} \Tr\bigl(\psi(s)b(X_0)\bigr)ds  + \Tr\biggl(\int_{0}^{T} \psi(s)\,dM_s^X\biggr) + \int_{0}^{T} \Tr\bigl(A(s)(X_s-X_0)\bigr) ds.
\end{align*}
Since $H,\psi \in L^2([0,T])$, the additional terms
$HH^\top$, $\psi\alpha\psi$, and $\cC_H(\psi)$ belong to $L^1([0,T])$, so the same Fubini argument is
applicable. Substituting this identity into the derived expression for $Y_T^{z,H}$, yields
$$Y_T^{z,H} = \Tr\biggl( uX_T + \int_{0}^{T} f(s)X_s \,ds \biggr).$$
Finally, if $e^L$ is a true martingale, then
\begin{equation*}
    E[e^{L_T} \mid \cF_t] = e^{L_t}.
\end{equation*}
Using the preceding expression for $Y_T^{z,H}$ yields \eqref{eq:transform-additional-brownian}.
\end{proof}

The next lemma provides a reformulation of the Riccati--Volterra equation \eqref{eq:riccati_1}, if $\widetilde{K} = \Id_d \otimes K$ admits a resolvent of the first kind.

\begin{lemma}\label{remark:psi-convolution}
	Suppose that $ \widetilde{K} = \Id_d \otimes K$ is measurable and $\widetilde{K} \in \loc^2(\R_+, \R^{d^2 \times d^2})$. Fix $u \in \bS_d(\C), f \in L^1([0,T], \bS_d(\C))$.
	If $\widetilde{K}$ admits a resolvent of the first kind $L$, then the Riccati--Volterra equation \eqref{eq:riccati_1} is equivalent to
	\begin{equation}\label{eq:psi-conv-wt-L}
		\int_{[0,T-t]} \vc\left( \psi(t+s)^\cplxtrans \right)^\cplxtrans L(ds)  = \vc\biggl(\Big( u + \int_{t}^{T} f(s) + B^\top \bigl(\psi(s)\bigr) - 2 \psi(s)\alpha \psi(s) \, ds\Big)^\cplxtrans \biggr)^\cplxtrans, \quad t \in [0,T].
	\end{equation}
	Using the notation for convolutions, this can be written as
	\begin{equation}\label{eq:convol-psi}
		\bigl(\vc\bigl(\psi(T - \cdot)^\cplxtrans\bigr)^\cplxtrans \ast L \bigr)(t) = \vc\biggl(\Big( u + \int_{T-t}^{T} f(s) + B^\top\bigl(\psi(s)\bigr) - 2\psi(s)\alpha \psi(s) \, ds\Big)^\cplxtrans \biggr)^\cplxtrans , \quad t \in [0,T].
	\end{equation}
\end{lemma}

\begin{proof}
	By applying the vectorization operator on the conjugated transposed Riccati--Volterra equation \eqref{eq:riccati_1}
	\begin{equation*}
		\psi(t)^\cplxtrans = K(T-t)^\top u^\cplxtrans + \int_{t}^{T} K(s-t)^\top \bigg(f(s) + B^\top(\psi(s)) - 2 \psi(s)\alpha \psi(s)\bigg)^\cplxtrans \, ds
	\end{equation*}
	and using the property $\vc(K^\top u) = (K^\top \otimes \Id_d) \vc(u) = (\Id_d \otimes K)^\top \vc(u)$, we see that it equals to
	\begin{align*}
		\vc(\psi(t)^\cplxtrans) = \widetilde{K}(T-t)^\top \vc(u^\cplxtrans) + \int_{t}^{T} \widetilde{K}(s-t)^\top \vc\Bigl(\bigl(f(s) + B^\top(\psi(s)) - 2 \psi(s)\alpha \psi(s)\bigr)^\cplxtrans\Bigr) \, ds.
	\end{align*}
	Since $L$ is the resolvent of the first kind of $\widetilde{K} \in \R^{d^2 \times d^2}$ this in turn implies
	\begin{align*}
		\int_{[0,T-t]}   \vc\bigl( \psi(t+s)^\cplxtrans \bigr)^\cplxtrans  L(ds) & = \int_{[0,T-t]} \vc(u^\cplxtrans)^\cplxtrans \widetilde{K}(T-(t+s)) L(ds)                                                                                                     \\
		                                                                         & \quad + \int_{[0,T-t]} \int_{t+s}^{T}  \vc\Bigl(\bigl(f(r) + B^\top(\psi(r)) - 2 \psi(r)\alpha \psi(r)\bigr)^\cplxtrans\Bigr)^\cplxtrans \widetilde{K}(r-(t+s)) \, dr \, L(ds) \\
		                                                                         & = \textbf{I} + \textbf{II}.
	\end{align*}
	For the first term we get
	\begin{align*}
		\textbf{I} =  \vc(u^\cplxtrans)^\cplxtrans \bigg( \int_{[0,T-t]} \widetilde{K}(T-(t+s)) L(ds) \bigg) = \vc(u^\cplxtrans)^\cplxtrans (\widetilde{K} \ast L)(T-t)  = \vc(u^\cplxtrans)^\cplxtrans .
	\end{align*}
	The second term, with an application of Fubinis theorem, equals
	\begin{align*}
		\textbf{II} & = \int_{[0,T-t]} \bigg( \int_{t+s}^{T} \vc\left(\big(f(r) + B^\top(\psi(r)) - 2 \psi(r)\alpha \psi(r)\big)^\cplxtrans\right)^\cplxtrans \widetilde{K}(r-(t+s))\, dr \bigg) L(ds) \\
		            & = \int_{0}^{T-t} \vc\Bigl(\big(f(t+u) + B^\top(\psi(t+u)) - 2 \psi(t+u)\alpha \psi(t+u)\big)^\cplxtrans\Bigr)^\cplxtrans \bigg( \int_{[0,u]} \widetilde{K}(u-s) L(ds)\bigg) du   \\
		            & = \int_{t}^{T} \vc\Bigl(\big(f(u) + B^\top(\psi(u)) - 2 \psi(u)\alpha \psi(u)\big)^\cplxtrans\Bigr)^\cplxtrans \, du.
	\end{align*}
	This shows \eqref{eq:psi-conv-wt-L}.
\end{proof}

The following theorem shows that, under mild additional assumptions on the kernel $\widetilde{K}$, the conditional Fourier--Laplace transform admits a representation which is affine in the past trajectory $\{X_s : s \leq t\}$ for every $t \in [0,T]$.

\begin{theorem}\label{thm:affine-on-past-path}
	Let $X = (X_t)_{t \in [0,T]}$ be an affine Volterra process on $\bS_d^+$. Let $\widetilde{K} = \Id_d \otimes K$ be continuous on $(0,T]$ with resolvent of the first kind $L$ with
	$\sup_{r \leq T} \|(\Delta_r \widetilde{K})\ast L\|_{TV} < \infty$. Fix $u \in \bS_d(\C)$ and $f \in L^1([0,T], \bS_d(\C))$. Let $\psi \in L^2([0,T], \bS_d(\C))$
	solve the Riccati--Volterra equation \eqref{eq:riccati_1} and let $Y$ be defined by \eqref{eq:Y-t}.

	\begin{enumerate}
		\item[(i)] it holds that $\vc(\psi(t - \cdot)^\cplxtrans)^\cplxtrans \ast L : [0,t] \to (\C^{d^2})^\ast$ is right-continuous and of bounded variation,
		      and that
		      \begin{align}\label{eq:Y-1-path-dependent}
			      \begin{split}
				      Y_t & = \bigg( \int_{[0,T]} \vc(\psi(s)^\cplxtrans)^\cplxtrans L(ds)\bigg)\vc(X_0) + \vc(\psi(t)^\cplxtrans)^\cplxtrans L(\{0\}) \vc(X_t)                                           \\
				          & \quad - \bigg( \int_{[0,t]} \vc(\psi(s)^\cplxtrans)^\cplxtrans L(ds) \bigg)\vc(X_0) + \int_{[0,t]} d\big(\vc(\psi(t-\cdot)^\cplxtrans)^\cplxtrans \ast L \big)(s)\vc(X_{t-s}) \\
				          & \quad + \Tr \biggl( \int_{t}^{T} \psi(s)b_0 ds - \int_{0}^{t} \Bigl(B^\top(\psi(s)) - 2\psi(s)\alpha \psi(s)\Bigr)X_s ds\biggr), \quad t \in [0,T].
			      \end{split}
		      \end{align}
		\item[(ii)] For all $t \in [0,T]$ let
		      \begin{equation}\label{eq:g_t}
			      g_t : [0,t] \to (\C^{d^2})^\ast, \quad g_t(r) = - \int_{(r,T-t+r]} \vc(\psi(t-r+s)^\cplxtrans)^\cplxtrans L(ds), \quad r \in [0,t].
		      \end{equation}
		      It then holds for all $t \in [0,T]$ that $g_t : [0,t] \to (\C^{d^2})^\ast$ is right-continuous and of bounded variation, and that
		      \begin{align}\label{eq:Y-2-path-dependent}
			      \begin{split}
				      Y_t & = -g_t(t)\vc(X_0) + \vc(\psi(t)^\cplxtrans)^\cplxtrans L(\{0\})\vc(X_t) + \int_{[0,t]} (dg_t(s))\vc(X_{t-s}) \\
				          & \quad + \Tr \biggl( \int_{0}^{t} f(s)X_s ds + \int_{t}^{T} \psi(s)b_0 ds \biggr), \quad t \in [0,T].
			      \end{split}
		      \end{align}
	\end{enumerate}
\end{theorem}

\begin{proof}
	Note first that for all $t \in [0,T]$ it holds
	\begin{equation}\label{eq:g-in-proof}
		g_t = \vc(\psi(t - \cdot)^\cplxtrans)^\cplxtrans \ast L - (\vc(\psi(T - \cdot)^\cplxtrans)^\cplxtrans \ast L)(T-t+\cdot).
	\end{equation}
	The proof that the function $\vc(\psi(t-\cdot)^\cplxtrans)^\cplxtrans \ast L : [0,t] \to (\C^{d^2})^\ast$ is right-continuous and of bounded variation follows analogously
	to the second part of the proof in \cite[Theorem 3.7]{ackermann_Inhomogeneous_2022} with the function $G : [0,T] \to (\C^{d^2})^\ast$ in the proof replaced by
	$G(s) = \vc\Bigl(\big(f(s) + B^\top(\psi(s)) - 2\psi(s) \alpha \psi(s)\big)^\cplxtrans \Bigr)^\cplxtrans$, $s \in [0,T]$.
	Since \eqref{eq:convol-psi} shows that $\vc(\psi(T - \cdot)^\cplxtrans)^\cplxtrans \ast L$ is continuous and of bounded variation, it then follows that for all
	$t \in [0,T]$ also $g_t : [0,t] \to (\C^{d^2})^\ast$ is right-continuous and of bounded variation.

	To derive the expressions \eqref{eq:Y-1-path-dependent} and \eqref{eq:Y-2-path-dependent} for $Y$, observe that we obtain from the definition
	\eqref{eq:Y-t} of $Y_0$ together with Lemma \ref{remark:psi-convolution} that
	\begin{align*}
		Y_0 & = \Tr\biggl(\Big( u + \int_{0}^{T} f(s) + B^\top(\psi(s)) - 2\psi(s) \alpha \psi(s) \, ds \Big)X_0 \biggr) + \Tr \biggl( \int_{0}^{T} \psi(s)b_0 \, ds\biggr)                            \\
		    & = \vc\biggl(\Big(u + \int_{0}^{T} f(s) + B^\top(\psi(s)) - 2\psi(s) \alpha \psi(s) \, ds \Big)^\cplxtrans \biggr)^\cplxtrans \vc(X_0) + \Tr \biggl( \int_{0}^{T} \psi(s)b_0 \, ds\biggr) \\
		    & = \big(\vc(\psi(T - \cdot)^\cplxtrans)^\cplxtrans \ast L\big)(T)\vc(X_0) + \Tr \biggl( \int_{0}^{T} \psi(s)b_0 \, ds\biggr).
	\end{align*}
	Inserting this in definition \eqref{eq:Y-t} of $Y$ yields for all $t \in [0,T]$
	\begin{align}\label{eq:path-dep-temp-1}
		\begin{split}
			Y_t & = \big(\vc(\psi(T - \cdot)^\cplxtrans)^\cplxtrans \ast L \big)(T)\vc(X_0) + \Tr \biggl( \int_{0}^{T} \psi(s)b_0 \, ds                                    \\
			    & \quad + \int_{0}^{t} \psi(s) \sqrt{X_s} dW_s Q + \int_{0}^{t} Q^\top dW_s^\top \sqrt{X_s}\psi(s) + 2 \int_{0}^{t} X_s\psi(s)\alpha\psi(s) \, ds\biggr)   \\
			    & = \big(\vc(\psi(T - \cdot)^\cplxtrans)^\cplxtrans \ast L \big)(T)\vc(X_0) + \Tr \biggl( \int_{0}^{t} \psi(s)b(X_s) \, ds + \int_{t}^{T} \psi(s)b_0 \, ds \\
			    & \quad + \int_{0}^{t} \psi(s) \sqrt{X_s} dW_s Q + \int_{0}^{t} Q^\top dW_s^\top \sqrt{X_s}\psi(s)                                                        \\
			    & \quad - \int_{0}^{t} \big(B^\top(\psi(s)) - 2\psi(s)\alpha \psi(s)\big)X_s \, ds\biggr).
		\end{split}
	\end{align}
	Let $Z = \int_{0}^{\cdot} b(X_s) \, ds + \int_{0}^{\cdot} (\sqrt{X_s} dW_s Q + Q^\top dW_s^\top \sqrt{X_s})$. It then holds for all $t \in [0,T]$
	\begin{align}\label{eq:path-dep-temp-2}
		\begin{split}
			\Tr\bigl( (\psi(t-\cdot) \ast dZ)(t)\bigr) = \Tr\biggl( \int_{0}^{t} \psi(s)b(X_s) \, ds + \int_{0}^{t} \psi(s)\Bigl(\sqrt{X_s} dW_s Q + Q^\top dW^\top_s\sqrt{X_s}\Bigr)\biggr).
		\end{split}
	\end{align}
	Since for all $t \in [0,T]$ the function $\vc(\psi(t-\cdot)^\cplxtrans)^\cplxtrans \ast L : [0,t] \to (\C^{d^2})^\ast$ is right-continuous and of bounded variation,
	Lemma \ref{lm:transform-of-F-dZ} shows that for all $t \in [0,T]$
	\begin{align}\label{eq:path-dep-temp-3}
		\begin{split}
			\Tr \bigl((\psi(t-\cdot) \ast dZ)(t)\bigr) & = \big(\vc(\psi(t-\cdot)^\cplxtrans)^\cplxtrans \ast L\big)(0)\vc(X_t) - \big(\vc(\psi(t-\cdot)^\cplxtrans)^\cplxtrans \ast L \big)(t)\vc(X_0) \\
			                                           & \quad + \big(d(\vc(\psi(t-\cdot)^\cplxtrans)^\cplxtrans \ast L \big) \ast \vc(X))(t)
		\end{split}
	\end{align}
	We combine \eqref{eq:path-dep-temp-1}, \eqref{eq:path-dep-temp-2} and \eqref{eq:path-dep-temp-3} to obtain for all $t \in [0,T]$
	\begin{align*}
		Y_t & = \big(\vc(\psi(T-\cdot)^\cplxtrans)^\cplxtrans \ast L\big)(T)\vc(X_0) + \vc(\psi(t)^\cplxtrans)^\cplxtrans L(\{0\})\vc(X_t) - \big(\vc(\psi(t-\cdot)^\cplxtrans)^\cplxtrans \ast L \big)(t)\vc(X_0) \\
		    & \quad + \big(d(\vc(\psi(t-\cdot)^\cplxtrans)^\cplxtrans \ast L \big) \ast \vc(X))(t)                                                                                                                 \\
		    & \quad + \Tr \biggl( \int_{t}^{T} \psi(s)b_0 \, ds - \int_{0}^{t} \Bigl(B^\top(\psi(s)) - 2\psi(s)\alpha \psi(s)\Bigr)X_s \, ds\biggr).
	\end{align*}
	This proves \eqref{eq:Y-1-path-dependent}. Furthermore, observe that by \eqref{eq:g-in-proof} we have for all $t \in [0,T]$
	\begin{equation*}
		\big(\vc(\psi(T-\cdot)^\cplxtrans)^\cplxtrans \ast L \big)(T)\vc(X_0) - \big(\vc(\psi(t-\cdot)^\cplxtrans)^\cplxtrans \ast L \big)(t)\vc(X_0) = -g_t(t)\vc(X_0)
	\end{equation*}
	and
	\begin{equation*}
		d(\vc(\psi(t-\cdot)^\cplxtrans)^\cplxtrans \ast L) = dg_t + \mu_t
	\end{equation*}
	where $\mu_t$ denotes the measure associated to $r \mapsto \big(\vc(\psi(T-\cdot)^\cplxtrans)^\cplxtrans \ast L \big)(T-t+r)$. It therefore follows for all $t \in [0,T]$
	\begin{align}\label{eq:Y-path-dep-w-g}
		\begin{split}
			Y_t & = -g_t(t)\vc(X_0) + \vc(\psi(t)^\cplxtrans)^\cplxtrans L(\{0\})\vc(X_t) + \big((dg_t) \ast \vc(X)\big)(t) + \big(\mu_t \ast \vc(X)\big)(t) \\
			    & \quad + \Tr \biggl( \int_{t}^{T} \psi(s)b_0 \, ds - \int_{0}^{t} \Bigl(B^\top(\psi(s)) - 2\psi(s)\alpha \psi(s)\Bigr)X_s \, ds\biggr).
		\end{split}
	\end{align}
	Note that by \eqref{eq:convol-psi} the measure $\mu_t$ for each $t \in [0,T]$ is given by
	\begin{align*}
		\mu_t\big((r,l]\big) = \vc \Biggl(\bigg( \int_{r}^{l} f(t-s) + B^\top(\psi(t-s)) - 2\psi(t-s)\alpha \psi(t-s) \, ds\bigg)^\cplxtrans \Biggr)^\cplxtrans, \quad r,l \in [0,t], \ r < l.
	\end{align*}
	It thus holds that
	\begin{align*}
		\int_{[0,t]} \mu_t(ds) \vc(X_{t-s}) & = \int_{0}^{t} \vc\Bigl( \bigl(f(t-s) + B^\top(\psi(t-s)) - 2\psi(t-s)\alpha \psi(t-s) \bigr)^\cplxtrans \Bigr)^\cplxtrans \vc(X_{t-s})\, ds        \\
		                                    & = \int_{0}^{t} \vc\Bigl(\bigl(f(r) + B^\top(\psi(r)) - 2\psi(r)\alpha \psi(r) \bigr)^\cplxtrans \Bigr)^\cplxtrans \vc(X_r)\, dr, \quad t \in [0,T].
	\end{align*}
	Rewriting this with the trace operator and substituting it into \eqref{eq:Y-path-dep-w-g} yields \eqref{eq:Y-2-path-dependent}.
\end{proof}

\begin{remark}\label{rem:affine-path-scalar-kernel}
	In the case that $K = k\Id_d$, as in Assumption~\ref{ass:admissible-parameters}, then $\widetilde{K} = k\Id_{d^2}$,
	and the resolvent of the first kind of $\widetilde{K}$ is simply given by $L\Id_{d^2}$, where $L$ denotes the scalar resolvent of $k$.
	Hence, the path-dependent representation in Theorem~\ref{thm:affine-on-past-path} can be written entirely in matrix notation.
	In particular, defining
	\begin{equation*}
		g_t(r) := - \int_{(r,T-t+r]} \psi(t-r+s)\,L(ds),\quad r \in [0,t],
	\end{equation*}
	part~(ii) becomes
	\begin{align*}
			      \begin{split}
				      Y_t  = \Tr\biggl(-g_t(t)X_0 + L(\{0\})\psi(t)X_t + \int_{[0,t]} X_{t-s} (dg_t(s)) + \int_{0}^{t} f(s)X_s ds + \int_{t}^{T} \psi(s)b_0 ds \biggr), \quad t \in [0,T].
			      \end{split}
	\end{align*}
\end{remark}

\section{A commodity model with Volterra--Wishart variance-covariance}\label{section:Commodity-models}

In this section, we extend the Gibson--Schwartz model by a Volterra--Wishart $2 \times 2$ variance-covariance-process $V$ and calculate its Fourier--Laplace transform, in spirit of \cite{schneider_Revisiting_2024}.

\subsection{The extended Gibson and Schwartz model}
Fix $T < \infty$, $\kappa > 0$, $\mu, \theta \in \R$, and deterministic initial value $S_0 > 0$, $\delta_0 \in \R$ and $V_0 \in \bS_2^+$. Throughout this section $K = k\Id_2$, where $k$ satisfies Assumptions~\ref{K-regularity-condition} and \ref{ass:k-PSD-conditions}, with $b_0 = b$ and $B(x) = Mx + xM^\top$.
We consider the extended Gibson--Schwartz model
\begin{equation}\label{eq:GS-model}
d\begin{pmatrix}
    \log S_t \\
    \delta_t
\end{pmatrix}
=
\begin{pmatrix}
    \mu-\delta_t-\frac{1}{2} V_{11,t} \\
    \kappa(\theta-\delta_t)
\end{pmatrix}dt
+ \sqrt{V_t}\,dZ_t,\quad t \in [0,T],
\end{equation}
where $Z$ is a two-dimensional Brownian motion and $V = (V_t)_{t \in [0,T]}$ is the $\bS_2^+$-valued Volterra--Wishart process
\begin{align}\label{eq:GS-volterra-wishart}
	\begin{split}
		V_t & = V_0 + \int_{0}^{t} K(t-s)\bigl( b + MV_s + V_s M^\top\bigr) ds + \int_{0}^{t} K(t-s) \bigl(\sqrt{V_s}dW_s Q + Q^\top dW_s^\top \sqrt{V_s} \bigr),
	\end{split}
\end{align}
with $Q,M \in \R^{2 \times 2}$, covariance parameter $ \alpha = Q^\top Q \in \bS_2^+$, $b \succeq k(0)(d-1)\alpha$.
Here, $W = \begin{pmatrix}
		W_{11} & W_{12} \\
		W_{21} & W_{22}
	\end{pmatrix}$ is a $2\times 2$-dimensional Brownian motion, and it can be correlated with $Z$ as follows. Define the correlation vector
\begin{equation*}
	\rho = \begin{pmatrix}
		\rho_1 \\
		\rho_2
	\end{pmatrix},
\end{equation*}
subject to
\begin{equation*}
	|\rho| = \sqrt{\rho^\top \rho} = \sqrt{\rho_1^2 + \rho_2^2} \, \leq 1,
\end{equation*}
and set
\begin{equation*}
	Z = \begin{pmatrix}
		Z_1 \\
		Z_2
	\end{pmatrix}
	= W \rho + \sqrt{1-\rho^\top \rho}\, \widetilde{W},
\end{equation*}
where $\widetilde{W} = \begin{pmatrix}
		\widetilde{W}_1 \\
		\widetilde{W}_2
	\end{pmatrix}$ is a Brownian motion independent of $W$.

Our aim is to calculate the conditional Fourier--Laplace transform
\begin{equation}\label{eq:GS-transform-objective}
\Phi_t(u_1,u_2,u_3) := E\left[ \exp\bigl( u_1 \log S_T + u_2 \delta_T - \Tr(u_3V_T) \bigr) \mid \cF_t \right],
\end{equation}
where $\bF = (\cF_t)_{t \in [0,T]}$ denotes the usual augmentation of the filtration generated by $(V,W,\widetilde{W})$. The natural domain is given by
\begin{equation*}
 u_1, u_2 \in i\R, \quad u_3 \in \bS_2^{+} + i\bS_2.
\end{equation*}
We first express $u_1 \log S_T + u_2 \delta_T$ using deterministic integrands. We then verify its exponential-affine
transform by a direct stochastic-exponential calculation. This avoids conditioning on the entire future
covariance path, which need not be independent of $\widetilde{W}$ for an arbitrary weak solution.

Let $e_1,e_2$ denote the canonical basis vectors of $\R^2$. From \eqref{eq:GS-model}, we have
\begin{align*}
d\log S_t &= \left(\mu - \delta_t - \frac{1}{2} V_{11,t} \right)dt + e_1^\top \sqrt{V_t} \, dZ_t,\\
d\delta_t &= \kappa(\theta - \delta_t)dt + e_2^\top \sqrt{V_t} \,dZ_t.
\end{align*}
Multiplication of the second equation by the integrating factor $e^{\kappa t}$ yields
\begin{equation*}
	d\bigl(e^{\kappa t} \delta_t\bigr) = \kappa \theta e^{\kappa t}\,dt + e^{\kappa t} e_2^\top \sqrt{V_t}\,dZ_t.
\end{equation*}
Consequently,
\begin{equation}\label{eq:GS-delta-solution}
\delta_t = \theta + e^{-\kappa t}(\delta_0 - \theta) + \int_{0}^{t} e^{-\kappa(t-s)} e_2^\top \sqrt{V_s}\,dZ_s.
\end{equation}
Integrating \eqref{eq:GS-delta-solution} over $[0,T]$ and applying the stochastic Fubini theorem yields
\begin{equation*}
	\int_{0}^{T} \delta_s \, ds = \frac{1-e^{-\kappa T}}{\kappa} \delta_0 + \theta \left(T - \frac{1-e^{-\kappa T}}{\kappa}\right) + \int_{0}^{T} \frac{1-e^{-\kappa(T-s)}}{\kappa} e_2^\top \sqrt{V_s}\,dZ_s.
\end{equation*}
On the other hand, we also have
\begin{equation}\label{eq:GS-log-price}
	\log S_T = \log S_0 + \mu T - \int_{0}^{T} \delta_s \, ds - \frac{1}{2} \int_{0}^{T} V_{11,s} \, ds + \int_{0}^{T} e_1^\top \sqrt{V_s} \, dZ_s,
\end{equation}
and combining \eqref{eq:GS-delta-solution}--\eqref{eq:GS-log-price}, we obtain
\begin{equation}\label{eq:GS-linear-combination}
u_1\log S_T + u_2 \delta_T = C_T + \int_{0}^{T} \chi(s) \sqrt{V_s} \, dZ_s - \frac{u_1}{2} \int_{0}^{T} V_{11,s} \, ds,
\end{equation}
where the deterministic parts equal
\begin{equation*}
	C_T = u_1\left[ \log S_0 + \mu T - \frac{1-e^{-\kappa T}}{\kappa} \delta_0 - \theta\left( T-\frac{1-e^{-\kappa T}}{\kappa}\right) \right] + u_2\left[\theta + e^{- \kappa T}(\delta_0 - \theta)\right]
\end{equation*}
and
\begin{equation}\label{eq:GS-chi}
\chi(s) =
\begin{pmatrix}
\chi_1(s) & \chi_2(s)
\end{pmatrix}
:=
\begin{pmatrix}
 u_1 & u_2 e^{-\kappa(T-s)} -\dfrac{u_1}{\kappa} \bigl(1 - e^{-\kappa(T-s)}\bigr)
\end{pmatrix}.
\end{equation}

The stochastic integral in \eqref{eq:GS-linear-combination} decomposes by definition of $Z$ as
\begin{align*}
\int_{0}^{T} \chi(s) \sqrt{V_s} \, dZ_s = \int_{0}^{T} \chi(s)\sqrt{V_s} dW_s \rho + \sqrt{1-\rho^\top \rho} \int_{0}^{T} \chi(s)\sqrt{V_s}\,d\widetilde{W}_s.
\end{align*}
So \eqref{eq:GS-transform-objective} equals now
\begin{align*}
	\Phi_t(u_1, u_2, u_3) := \exp(C_T) &E \biggl[ \exp\biggl(\int_{0}^{T} \chi(s)\sqrt{V_s} dW_s \rho + \sqrt{1-\rho^\top \rho} \int_{0}^{T} \chi(s)\sqrt{V_s}\,d\widetilde{W}_s \\
	& \qquad - \frac{u_1}{2} \int_{0}^{T} V_{11,s} \, ds - \Tr\bigl(u_3 V_T\bigr)\biggr) \biggm\vert \cF_t \biggr]
\end{align*}
The coefficients $\chi$ are purely imaginary. We now rewrite the Fourier--Laplace transform in a form to which Proposition~\ref{prop:additional-brownian-functional} can be applied, while keeping both Brownian terms in the final stochastic exponential.

Since $u_1,u_2 \in i\R$, there exists a deterministic function $\eta : [0,T] \to \R^{1 \times 2}$ such that
$\chi(s) = i\eta(s)$.
Define furthermore
\begin{equation*}
H(s) := \eta(s)^\top\rho^\top \in \R^{2 \times 2}, \qquad N_t := \int_{0}^{t} \Tr\bigl(H(s)^\top \sqrt{V_s} \, dW_s\bigr).
\end{equation*}
By cyclicity of the trace it holds
\begin{equation*}
iN_t = \int_{0}^{t} \chi(s)\sqrt{V_s} dW_s \rho.
\end{equation*}
Thus, the part of the Fourier--Laplace transform driven by the Brownian motion $W$ is precisely the additional Brownian functional covered by Proposition~\ref{prop:additional-brownian-functional} with $z = i$.
For the remaining terms, set
\begin{equation}\label{eq:GS-f0}
f_0(s) := \frac{1}{2} \left(u_1e_1e_1^\top -(1-\rho^\top \rho)\chi(s)^\top\chi(s)\right),
\end{equation}
since then it follows that
\begin{equation*}
	- \int_{0}^{T} \Tr\bigl(f_0(s)V_s\bigr) = - \frac{u_1}{2} \int_{0}^{T} V_{11,s} \, ds + \frac{1-\rho^\top \rho}{2} \int_{0}^{T} \chi(s)V_s\chi(s)^\top.
\end{equation*}

Substituting those expressions in the Fourier--Laplace transform we arrive at
\begin{align*}
    \Phi_t(u_1,u_2,u_3) = \exp(C_T) &E\biggl[ \exp\biggl(iN_T - \int_{0}^{T} \Tr\bigl(f_0(s)V_s\bigr)\,ds - \Tr(u_3 V_T) \\
        & \qquad+ \sqrt{1-\rho^\top\rho} \int_{0}^{T} \chi(s)\sqrt{V_s}\,d\widetilde W_s - \frac{1-\rho^\top\rho}{2} \int_{0}^{T} \chi(s)V_s\chi(s)^\top\,ds \biggr) \biggm| \cF_t \biggr].
\end{align*}
The first line of the exponent is of the form treated by Proposition~\ref{prop:additional-brownian-functional}, while the last two terms form a stochastic exponential associated with the independent Brownian motion $\widetilde{W}$.

Notice that
\begin{equation*}
\re f_0(s) = \frac{1-\rho^\top \rho}{2} \eta(s)^\top\eta(s) \succeq 0, \quad s \in [0,T].
\end{equation*}
For a solution of the following Riccati equation, Proposition~\ref{prop:additional-brownian-functional} gives us a local-martingale candidate with
\begin{equation*}
z = i, \qquad u = u_3, \qquad f = f_0, \qquad H(s) = \eta(s)^\top\rho^\top
\end{equation*}
and the associated Riccati--Volterra equation therefore is
\begin{align}\label{eq:GS-riccati-preliminary}
	\begin{split}
		\psi(t) = u_3K(T-t) &+ \int_{t}^{T} \Bigl( f_0(s) + M^\top\psi(s) + \psi(s)M
+i\bigl[ H(s)Q\psi(s) +\psi(s)Q^\top H(s)^\top \bigr] \\
&\qquad  +\frac{1}{2} H(s)H(s)^\top - 2\psi(s)\alpha\psi(s) \Bigr) K(s-t)\,ds.
	\end{split}
\end{align}
Since $iH(s) = \chi(s)^\top\rho^\top$ and $H(s)H(s)^\top = -(\rho^\top\rho)\chi(s)^\top\chi(s)$,
we obtain
\begin{equation}\label{eq:GS-f}
f(s) := f_0(s) + \frac{1}{2}H(s)H(s)^\top = \frac{1}{2}\left(u_1e_1e_1^\top - \chi(s)^\top\chi(s) \right).
\end{equation}
Also, in particular, $\re f(s) = \frac{1}{2} \eta(s)^\top \eta(s) \succeq 0$ and in this case $f$ is given explicitly by
\begin{equation*}
f(s) = \frac{1}{2}
\begin{pmatrix}u_1-\chi_1(s)^2
&-\chi_1(s)\chi_2(s)
\\[1mm]
-\chi_1(s)\chi_2(s)
&-\chi_2(s)^2
\end{pmatrix}.
\end{equation*}
Furthermore, $i\bigl[ H(s)Q\psi(s) + \psi(s)Q^\top H(s)^\top \bigr] = \chi(s)^\top\rho^\top Q\psi(s)+ \psi(s)Q^\top\rho\,\chi(s)$.
Consequently, \eqref{eq:GS-riccati-preliminary} reduces to
\begin{align}\label{eq:GS-full-riccati}
	\begin{split}
		\psi(t) = u_3K(T-t) &+ \int_{t}^{T} \Bigl(f(s) + M^\top\psi(s) + \psi(s)M \\
&\qquad  +\chi(s)^\top\rho^\top Q\psi(s) + \psi(s)Q^\top\rho\chi(s) - 2\psi(s)\alpha\psi(s) \Bigr) K(s-t)\,ds.
	\end{split}
\end{align}
Equivalently, \eqref{eq:GS-full-riccati} can be written in the compact form
\begin{align}\label{eq:GS-full-riccati-compact}
	\begin{split}
		\psi(t) = u_3K(T-t) + \int_{t}^{T} \Bigl(f(s) + M^\chi(s)^\top\psi(s) + \psi(s)M^\chi(s) - 2\psi(s)\alpha\psi(s) \Bigr) K(s-t)\,ds
	\end{split}
\end{align}
with $M^\chi(s) := M+Q^\top\rho\,\chi(s)$.
For a given solution $\psi \in L^2([0,T],\bS_2(\C))$ of \eqref{eq:GS-full-riccati}, let $Y^{i,H}$ be the process from Proposition~\ref{prop:additional-brownian-functional}. Then, $\exp(iN-Y^{i,H})$ is a local martingale. 
Moreover, by the terminal condition in Proposition~\ref{prop:additional-brownian-functional},
\begin{equation*}
	Y_T^{i,H} = \Tr \bigl(u_3 V_T\bigr) + \int_{0}^{T} \Tr \bigl(f_0(s) V_s\bigr) \, ds.
\end{equation*}
Hence,
\begin{equation*}
	\exp(iN_T - Y_T^{i,H}) = \exp \biggl(iN_T - \Tr \bigl(u_3 V_T\bigr) - \int_{0}^{T} \Tr\bigl(f_0(s)V_s\bigr) \, ds\biggr).
\end{equation*}
Comparing this with the last derived representation of $\Phi_t$ shows explicitly that the only remaining factor is the stochastic exponential associated with $\widetilde{W}$. Indeed, define
\begin{equation*}
    \cE_t^\perp := \exp\biggl( \sqrt{1-\rho^\top\rho} \int_{0}^{t} \chi(s)\sqrt{V_s} \, d\widetilde{W}_s - \frac{1-\rho^\top\rho}{2} \int_{0}^{t} \chi(s)V_s\chi(s)^\top \, ds \biggr).
\end{equation*}
Then the Fourier--Laplace exponential can be written as
\begin{equation*}
    \exp\bigl( u_1\log S_T + u_2\delta_T - \Tr(u_3 V_T) \bigr) = \exp(C_T) \exp\bigl(iN_T-Y_T^{i,H}\bigr) \cE_T^\perp.
\end{equation*}
Thus, the original Fourier--Laplace exponential has been decomposed into a deterministic factor and two stochastic exponentials.
For completeness, the process $Y^{i,H}$ is given by
\begin{align*}
	Y^{i,H}_t &= Y_0^{i,H} + \Tr \biggl( \int_{0}^{t} \psi(s) \sqrt{V_s} dW_s Q + \int_{0}^{t} Q^\top dW_s^\top \sqrt{V_s} \psi(s) \\
	& \quad + \int_{0}^{t} V_s \Bigl( 2\psi(s)\alpha \psi(s) + \frac{1}{2} (\rho^\top \rho) \chi(s)^\top \chi(s) - \chi(s)^\top \rho^\top Q \psi(s) - \psi(s)Q^\top \rho \chi(s) \Bigr) ds \biggr),
\end{align*}
with 
\begin{align*}
	Y_0^{i,H} = \Tr \biggl( u_3 V_0 + \int_{0}^{T} \psi(s)b + V_0 \bigl[f(s) + M^\chi(s)^\top \psi(s) + \psi(s)M^\chi(s) - 2\psi(s)\alpha \psi(s)\bigr]ds\biggr).
\end{align*}

The local martingale part of $\exp(iN-Y^{i,H})$ is driven by $W$, whereas $\cE^\perp$ is driven by $\widetilde{W}$. Since $W$ and $\widetilde{W}$ are independent Brownian motions, their quadratic covariation vanishes. Hence, It\^o's product formula implies that
\begin{equation*}
    M_t := \exp\bigl(iN_t-Y_t^{i,H}\bigr) \cE_t^\perp
\end{equation*}
is again a local martingale.
Using the definitions of $N$ and $Z$ this can equivalently be written as
\begin{equation*}
    M_t = \exp\biggl( \int_{0}^{t} \chi(s)\sqrt{V_s}\,dZ_s - \frac{1-\rho^\top\rho}{2} \int_{0}^{t} \chi(s)V_s\chi(s)^\top\,ds - Y_t^{i,H} \biggr).
\end{equation*}
Thus, once $M$ is shown to be a true martingale, its terminal value satisfies
\begin{equation*}
    \exp(C_T)M_T = \exp\bigl( u_1\log S_T + u_2\delta_T - \Tr(u_3 V_T) \bigr).
\end{equation*}
The martingale property then yields the conditional transform
\begin{equation}\label{eq:GS-final-transform}
    \Phi_t(u_1,u_2,u_3)
    =
    \exp(C_T)M_t
    =
    \exp\Biggl(
        C_T
        + \int_0^t \chi(s)\sqrt{V_s}\,dZ_s
        - \frac{1-\rho^\top\rho}{2}
          \int_0^t \chi(s)V_s\chi(s)^\top\,ds
        - Y_t^{i,H}
    \Biggr).
\end{equation}
In particular,
\begin{equation*}
    E\Bigl[
        \exp\bigl(
            u_1\log S_T
            + u_2\delta_T
            - \Tr(u_3V_T)
        \bigr)
    \Bigr]
    =
    \exp\bigl(C_T-Y_0^{i,H}\bigr).
\end{equation*}
It therefore remains to establish existence of a solution to \eqref{eq:GS-full-riccati} and to verify that the local martingale $M$ is a true martingale.

\begin{assumption}\label{K-regularity-in-resolvent}
    For every $h \in [0,1]$, the shifted kernel $\Delta_h k$ admits a resolvent
    of the first kind $L_h$. We assume that $L_h$ is nonnegative and that $s \mapsto L_h([s,(s+t) \wedge T])$ is nonincreasing for every $t \geq 0$. Denote by $L$ the resolvent of the first kind of $k$.
\end{assumption}

\begin{theorem}\label{thm:GS-model-affine}
	Fix $T < \infty$. Assume that $K = k\Id_2$ satisfies Assumptions~\ref{K-regularity-condition}, \ref{ass:k-PSD-conditions}, and \ref{K-regularity-in-resolvent}.
	\begin{itemize}
		\item[(i)] The Volterra--Wishart system \eqref{eq:GS-model}--\eqref{eq:GS-volterra-wishart} admits a continuous weak solution $(\log S,\delta,V)$ with values in $\R^2 \times \bS_2^+$ for any initial condition $(\log S_0,\delta_0,V_0) \in\R^2\times\bS_2^+$. The paths of $V$ are Hölder continuous of any order less than $\gamma/2$, where $\gamma$ is the constant associated with $k$ and $T$ in Assumption~\ref{K-regularity-condition}.
		\item[(ii)] Let $u_3 \in\bS_2^+$ and let $f_L \in L^1([0,T],\bS_2^+)$ be real-valued. Then the Riccati--Volterra equation
		\begin{equation}\label{eq:GS-real-riccati}
		\psi(t) = u_3k(T-t)+ \int_{t}^{T} k(s-t)\bigl(f_L(s)+M^\top\psi(s)+\psi(s)M-2\psi(s)\alpha\psi(s)\bigr)\,ds
		\end{equation}
		has a unique solution $\psi \in L^2([0,T],\bS_2^+)$. Moreover, the corresponding conditional Laplace functional is exponential-affine, i.e. if $Y^L$ denotes the process in Theorem~\ref{thm:main-thm} associated with $u = u_3$ and $f = f_L$, then
		\begin{equation}\label{eq:GS-Laplace-transform}
		E\biggl[\exp\Bigl(- \Tr(u_3V_T) - \int_{0}^{T}\Tr(f_L(s)V_s) \, ds\Bigr) \Bigm| \cF_t \biggr] = \exp(-Y_t^L), \quad t \in [0,T].
		\end{equation}
		In particular, for $f_L\equiv0$ this gives the conditional Laplace transform $\Phi_t(0,0,u_3)$.
		\item[(iii)] Assume in addition that $Q \in \GL_2$ and $k \in C^2([0,T])$. Let $u_1,u_2 \in i\R$ and $u_3 \in \bS_2^{+} + i\bS_2$, and let $\chi$, $f$ and $M^\chi$ be given by \eqref{eq:GS-chi}, \eqref{eq:GS-f} and \eqref{eq:GS-full-riccati-compact}. Then the Riccati--Volterra equation \eqref{eq:GS-full-riccati-compact} has a unique global solution
		\begin{equation*}
		\psi \in L^2([0,T],\bS_2(\C)),\qquad \re\psi(t)\succeq0,\quad t \in [0,T].
		\end{equation*}
		Moreover, the exponential-affine Fourier--Laplace transform formula \eqref{eq:GS-final-transform} holds on $[0,T]$.
	\end{itemize}
\end{theorem}

\begin{proof}
	Part (i) follows from Theorem~\ref{thm:weak-solutions}. Once $V$ and the driving matrix Brownian motion $W$ have been constructed, the equations for $(\log S,\delta)$ are linear stochastic differential equations and therefore define a continuous weak solution of the complete system.

For (iii), the global existence and the property $\re\psi \succeq 0$ are shown in Lemma~\ref{lem:GS-exsitence-riccati} below. We may consequently apply Proposition~\ref{prop:additional-brownian-functional}. In particular, $\exp\bigl(iN-Y^{i,H}\bigr)$ is a local martingale.
Define
\begin{equation*}
\cE_t^\perp := \exp\left(\sqrt{1-\rho^\top\rho} \int_{0}^{t} \chi(s)\sqrt{V_s}\,d\widetilde{W}_s -\frac{\bigl(1-\rho^\top\rho\bigr)}{2} \int_{0}^{t} \chi(s)V_s\chi(s)^\top\,ds\right).
\end{equation*}
The process $\sqrt{1-\rho^\top\rho} \int_{0}^{\cdot}  \chi(s)\sqrt{V_s}\,d\widetilde{W}_s$ is a continuous complex local martingale with quadratic variation
\begin{equation*}
\bigl(1-\rho^\top\rho\bigr) \int_{0}^{t} \chi(s)V_s\chi(s)^\top\,ds,
\end{equation*}
so $\cE^\perp$ is its stochastic exponential and hence a local martingale. The local martingale part of $\exp\bigl(iN-Y^{i,H}\bigr)$ is driven by $W$, whereas $\cE^\perp$ is driven by $\widetilde{W}$. Since $W$ and $\widetilde{W}$ are independent Brownian motions, their quadratic covariation vanishes. Consequently, the local martingale parts of $\exp(iN-Y^{i,H})$ and $\cE^\perp$ have zero quadratic covariation.
It\^o's product formula therefore implies that their product
\begin{equation*}
M_t := \exp\bigl(iN_t-Y_t^{i,H}\bigr)\cE_t^\perp
=\exp\Biggl( \int_{0}^{t}\chi(s)\sqrt{V_s}\,dZ_s -\frac{1-\rho^\top\rho}{2} \int_{0}^{t}\chi(s)V_s\chi(s)^\top\,ds -Y_t^{i,H}\Biggr)
\end{equation*}
is again a local martingale. It remains to verify that $M$ is also a true martingale. As in Theorem~\ref{thm:affine-on-past-path}(ii),
	\begin{align*}
		\begin{split}
			Y_t^{i,H} & = -g_t(t)\vc(V_0) + \vc(\psi(t)^\cplxtrans)^\cplxtrans L(\{0\})\vc(V_t) + \int_{[0,t]} (dg_t(s))\vc(V_{t-s}) \\
			    & \quad + \Tr \biggl( \int_{0}^{t} f_0(s)V_s \, ds + \int_{t}^{T} \psi(s)b_0 \, ds \biggr), \quad t \in [0,T].
		\end{split}
	\end{align*}
	Taking real parts, the sign properties of the first three terms follow
from $\re\psi\succeq0$, the nonnegativity of $L$, and
Assumption~\ref{K-regularity-in-resolvent}, by the same argument as in the proof of 
\cite[Proposition 4.3]{ackermann_Inhomogeneous_2022}.
Consequently,
\begin{equation*}
	-\re g_t(t)\vc(V_0) \geq 0,\quad \re\bigl(\vc(\psi(t)^\cplxtrans)^\cplxtrans L(\{0\})\vc(V_t)\bigr) \geq 0, \quad \int_{[0,t]} d(\re g_t(s))\vc(V_{t-s}) \geq 0.
\end{equation*}
Moreover, $b_0 \succeq 0$, $V_s \succeq 0$, $\re\psi(s) \succeq 0$, and by \eqref{eq:GS-f0},
\begin{equation*}
\re f_0(s) = \frac{1-\rho^\top\rho}{2}\eta(s)^\top\eta(s)\succeq0.
\end{equation*}
Consequently,
\begin{equation}\label{eq:GS-Y-lower-bound}
\re Y_t^{i,H} \geq \int_{0}^{t} \Tr\bigl(\re f_0(s)V_s\bigr)\,ds = \frac{1-\rho^\top\rho}{2} \int_{0}^{t} \eta(s)V_s\eta(s)^\top\,ds.
\end{equation}
Since $\chi = i\eta$, we have
$$\chi(s)V_s\chi(s)^\top = -\eta(s)V_s\eta(s)^\top,$$
and both stochastic integrals in the exponent of $M_t$ are purely imaginary. Hence, \eqref{eq:GS-Y-lower-bound} gives
\begin{align*}
\log|M_t| &= -\re Y_t^{i,H} - \frac{1-\rho^\top\rho}{2} \int_{0}^{t}\chi(s)V_s\chi(s)^\top\,ds\\
&= -\re Y_t^{i,H} + \frac{1-\rho^\top\rho}{2} \int_{0}^{t}\eta(s)V_s\eta(s)^\top\,ds \leq 0.
\end{align*}
Thus, $|M_t| \leq 1$ for all $t \in [0,T]$, and $M$ is a uniformly integrable martingale.
Finally, Proposition~\ref{prop:additional-brownian-functional} gives
$$Y_T^{i,H} = \Tr\biggl(u_3V_T+ \int_{0}^{t} f_0(s)V_s\,ds\biggr).$$
Using \eqref{eq:GS-linear-combination} and \eqref{eq:GS-f0}, we therefore obtain
$$\exp(C_T)M_T = \exp\bigl(u_1\log S_T+u_2\delta_T-\Tr(u_3V_T)\bigr).$$
Since $C_T$ is deterministic, the martingale property of $M$ yields
$$\Phi_t(u_1,u_2,u_3) = \exp(C_T)M_t,$$
which is precisely \eqref{eq:GS-final-transform}.

	For (ii), Lemma~\ref{lem:real-global-riccati} below gives the unique global solution with $\psi(t) \succeq 0$. We may therefore apply Theorem~\ref{thm:main-thm}. It remains to verify that the local martingale $\exp(-Y^L)$ is a true martingale. The path representation in Theorem~\ref{thm:affine-on-past-path}(ii), together with $\psi \succeq 0$, $f_L \succeq 0$, $V \succeq 0$, $b_0 \succeq 0$, the nonnegative resolvent $L$, and Assumption~\ref{K-regularity-in-resolvent}, gives
	$$Y_t^L \geq 0, \quad t \in [0,T],$$
	by exactly the same sign argument used in part (iii) for the real part of $Y^{i,H}$. Hence, $0 < \exp(-Y_t^L) \leq 1$, so the local martingale is uniformly integrable and \eqref{eq:GS-Laplace-transform} follows.
\end{proof}

\begin{corollary}\label{cor:GS-uniqueness}
    Under the assumptions of Theorem~\ref{thm:GS-model-affine} and the additional assumption that $Q \in \GL_2$ and $k \in C^2([0,T])$, any two continuous weak solutions of \eqref{eq:GS-model}--\eqref{eq:GS-volterra-wishart} have the same law.
\end{corollary}

\begin{proof}
    The argument follows the standard uniqueness argument for affine
    Volterra processes; see \cite[Theorem 6.1]{abijaber_Affine_2019} and \cite[Corollary 4.4]{ackermann_Inhomogeneous_2022}.
    Let $(\log S,\delta,V)$ be a continuous weak solution and fix
    $n \in \N$, $0 \leq t_1 < \cdots < t_n \leq T$, together with $u_{1,j},u_{2,j} \in i\R$ and
    $U_j \in \bS_2^+$, $j = 1, \ldots, n$.
    As in the proof of \cite[Corollary 4.4]{ackermann_Inhomogeneous_2022},
    by continuity of $V$ there exists a sequence
    $f_N \in C([0,T],\bS_2^+)$ such that
    \begin{equation*}
        \int_{0}^{T} \Tr(f_N(s)V_s) \, ds \rightarrow \sum_{j = 1}^{n} \Tr(U_jV_{t_j}) \quad \text{ a.s.}
    \end{equation*}
    Moreover, by linearity of \eqref{eq:GS-model},
    \begin{equation*}
        \sum_{j = 1}^{n} \bigl(u_{1,j} \log S_{t_j} + u_{2,j} \delta_{t_j} \bigr)
    \end{equation*}
    admits the same representation as in \eqref{eq:GS-linear-combination}, with deterministic
    piecewise defined coefficients.
    For every $N \in \N$, the corresponding Riccati--Volterra equation admits a unique global solution by the same argument as in Lemma~\ref{lem:GS-exsitence-riccati}, and the associated local martingale is a true martingale by the argument of Theorem~\ref{thm:GS-model-affine}. Hence, the corresponding Fourier--Laplace transform
    $$E\biggl[ \exp\biggl( \sum_{j = 1}^{n} \bigl(u_{1,j} \log S_{t_j} + u_{2,j} \delta_{t_j} \bigr) - \int_{0}^{T} \Tr(f_N(s)V_s)\,ds \biggr) \biggr]$$
    is uniquely determined by the model parameters. Letting $N \to \infty$ and using dominated convergence, we conclude
    that
    $$E\biggl[ \exp\biggl(\sum_{j = 1}^{n} \bigl( u_{1,j}\log S_{t_j} + u_{2,j} \delta_{t_j} - \Tr(U_j V_{t_j}) \bigr) \biggr) \biggr]$$
    is uniquely determined. Thus, any two continuous weak solutions
    have the same finite-dimensional distributions. Since their paths
    are continuous, their laws on
    $C([0,T],\R^2\times\bS_2^+)$ coincide.
\end{proof}

\begin{lemma}\label{lem:regular-first-kind-resolvent}
	Let $T < \infty$ and let $k \in C^2([0,T])$ satisfy $k(0) > 0$. Then its resolvent of the first kind has, on $[0,T]$, the form
	\begin{equation*}
	L(ds) = \ell_0\delta_0(ds) + \ell(s) \, ds, \quad \ell_0 = \frac{1}{k(0)},
	\end{equation*}
	for a function $\ell \in C^1([0,T])$.
\end{lemma}
\begin{proof}
    Consider the linear Volterra equation of the second kind
    \begin{equation}\label{eq:resolvent-density-equation}
        \ell(t) = - \frac{k'(t)}{k(0)^2} - \frac{1}{k(0)} \int_{0}^{t} k'(t-s)\ell(s)\,ds.
    \end{equation}
    Since $k' \in C^1([0,T])$, standard linear Volterra theory yields a unique continuous solution $\ell$, (see, for instance \cite[Theorem 2.3.5]{gripenberg_Volterra_1990}).
    Moreover, differentiating the convolution term in \eqref{eq:resolvent-density-equation} shows that $\ell \in C^1([0,T])$.
    Now define the measure
	$$ \widehat{L}(ds) := \frac{1}{k(0)}\delta_0(ds) + \ell(s)ds.$$
	We claim that $\widehat{L}$ is the resolvent of the first kind of $k$.
    Notice first that
    $$(k \ast \widehat{L})(t) = \frac{k(t)}{k(0)} + \int_{0}^{t} k(t-s)\ell(s)\,ds.$$
    Differentiating and using \eqref{eq:resolvent-density-equation} gives
    $$\frac{d}{dt}(k \ast \widehat{L})(t) = \frac{k'(t)}{k(0)} + k(0) \ell(t) + \int_{0}^{t} k'(t-s)\ell(s)\,ds = 0.$$
	By \eqref{eq:resolvent-density-equation}, the right-hand side is zero. Hence, $k \ast \widehat{L}$ is constant on $[0,T]$. For $t = 0$ we have $(k \ast \widehat{L})(0) = \frac{k(0)}{k(0)} = 1$ and therefore $k \ast \widehat{L} = 1$. Since the kernel and the measure are scalar-valued, convolution is commutative, and thus also $\widehat{L} \ast k \equiv 1$. Consequently, $\widehat{L}$ is a resolvent of the first kind of $k$.
	By uniqueness \cite[Theorem 5.5.2]{gripenberg_Volterra_1990} therefore yields $L = \widehat{L}$.
\end{proof}

\begin{lemma}\label{lem:GS-exsitence-riccati}
	Fix $T < \infty$ and assume $K = k\Id_2$ with $k$ satisfying Assumptions \ref{K-regularity-condition} and \ref{ass:k-PSD-conditions}. Assume in addition that $Q \in \GL_2$ and $k \in C^2([0,T])$. Let $u_1,u_2 \in i\R$, $u_3 \in \bS_2^+ + i\bS_2$, and define $\chi$ and $f$ by \eqref{eq:GS-chi} and \eqref{eq:GS-f}. Then the Riccati--Volterra equation \eqref{eq:GS-full-riccati-compact} has a unique global solution $\psi \in L^2([0,T],\bS_2(\C))$. Moreover,
	$$\re\psi(t) \succeq 0, \quad t \in [0,T]$$.
\end{lemma}

\begin{proof}
	Set $\widetilde{\psi}(t) = \psi(T-t)$, $t \in [0,T]$
    and define the time-dependent generalized Riccati map $R:[0,T] \times \bS_2(\C) \to \bS_2(\C)$
	by
	\begin{equation*}
		R(t,u) := f(T-t) + M^\chi(T-t)^\top u + u M^\chi(T-t) - 2u \alpha u, \quad t \in [0,T],\  u \in \bS_2(\C).
	\end{equation*}
	Then the time-reversed Riccati--Volterra equation
	\eqref{eq:GS-full-riccati-compact} takes the form
	\begin{equation*}
		\widetilde{\psi}(t) = u_3k(t) + \int_{0}^{t} k(t-s) R(s,\widetilde{\psi}(s)) \, ds, \quad t \in [0,T].
	\end{equation*}
    Since $k$ is locally square-integrable, $f$ and $M^\chi$ are continuous on finite intervals, and the map $u \mapsto u\alpha u$ is locally Lipschitz, Theorem~\ref{thm:non-continuable-solutions} gives a unique non-continuable solution $(\widetilde{\psi},\tau_{\max})$.

	We first prove invariance on $\bS_2^{+} + i\bS_2$. Decompose first $\widetilde{\psi} = P + iY$, with $P,Y \in \bS_2$,
	and, since $\chi = i\eta$ with $\eta$ being real-valued, we also write $M^\chi = M + iC$, with $C = Q^\top \rho \eta.$
	The real part of the generalized Riccati map thus equals
	\begin{align*}
	\re R(t,P+iY) &= \re f(T-t) + M^\top P+ PM - C^\top Y - YC\\
	&\quad - 2P\alpha P + 2Y\alpha Y.
	\end{align*}
	Let $Z \in \bS_2^+$ satisfy $\Tr(PZ) = 0$. Since $P,Z \succeq 0$, this in particular implies $PZ = ZP = 0$. Hence, all terms containing $P$ vanish when paired with $Z$. From \eqref{eq:GS-f} and $\chi = i\eta$ we have
	\begin{equation*}
	\re f = \frac{1}{2}\eta^\top\eta.
	\end{equation*}
	Using $\alpha = Q^\top Q$ and the Frobenius inner product, we therefore obtain
	\begin{align*}
	\langle\re R(t,P+iY), Z\rangle &= \frac{1}{2}\|\eta Z^{1/2}\|^2 - 2\langle QYZ^{1/2}, \rho\eta Z^{1/2}\rangle + 2\|QYZ^{1/2}\|^2 \\
	&= 2\|QYZ^{1/2} - \frac{1}{2}\rho\eta Z^{1/2}\|^2 + \frac{1}{2}(1-\rho^\top\rho) \|\eta Z^{1/2}\|^2 \\
	&\geq 0,
	\end{align*}
	since $\rho^\top\rho \leq 1$. Thus, the Riccati map is inward pointing at every boundary point of $\bS_2^{+} + i\bS_2$. Lemma~\ref{lem:volterra-cone-invariance} applied to the closed convex set $\bS_2^{+} + i\bS_2$ yields
	\begin{equation}\label{eq:invariance-complex-riccati}
	P(t) \succeq 0, \quad t < \tau_{\max}.
	\end{equation}

	It remains to exclude finite-time explosion. By Lemma~\ref{lem:regular-first-kind-resolvent}, the scalar resolvent of the first kind has the form
	\begin{equation*}
	L(ds) = \ell_0 \delta_0(ds) + \ell(s) ds, \qquad \ell_0 = k(0)^{-1} > 0, \quad \ell \in C^1([0,T]).
	\end{equation*}
	The time-reversed Riccati equation can be written as
	\begin{equation*}
	\widetilde{\psi}(t) = u_3k(t) + \int_{0}^{t} k(t-s)R(s,\widetilde{\psi}(s)) \, ds.
	\end{equation*}
	Convolving with the resolvent $L$ gives
	\begin{equation}\label{eq:complex-riccati-resolvent-form}
	\ell_0\widetilde{\psi}(t) + (\ell \ast \widetilde{\psi})(t) = u_3 + \int_{0}^{t} R(s,\widetilde{\psi}(s)) \, ds.
	\end{equation}
	Since $\ell \in C^1$ and $\widetilde{\psi}$ is continuous on compact subintervals of $[0,\tau_{\max})$, \eqref{eq:complex-riccati-resolvent-form} implies that $\widetilde{\psi}$ is locally absolutely continuous.
	
	The linear term of the Riccati equation can be controlled by standard estimate, whereas simple estimates of the quadratic term would lead to quadratic growth. We consequently use the invertibility of $Q$ to normalize the quadratic term.
	Set 
	\begin{equation*}
		\Psi(t) := Q\widetilde{\psi}(t)Q^\top = A(t) + iB(t), \quad A(t),B(t) \in \bS_2,
	\end{equation*}
	so that the quadratic term takes the canonical form $-2\Psi^2$. By \eqref{eq:invariance-complex-riccati}, we have $A(t) \succeq 0$. Let 
	\begin{equation*}
	\widehat{M} := (Q^{-1})^\top M Q^\top, \quad \widetilde{\eta}(t) := \eta(T-t), \quad D(t) := \rho\widetilde{\eta}(t) Q^\top,
	\end{equation*}
	and $\widehat{f}(t) := Qf(T-t)Q^\top$. Then
	\begin{equation*}
		Q M^\chi(T-t)^\top \widetilde{\psi}(t)Q^\top + Q \widetilde{\psi}(t) M^\chi(T-t) Q^\top = (\widehat{M} + iD(t))^\top \Psi(t) + \Psi(t)(\widehat{M} + iD(t)).
	\end{equation*}
	Consequently, after multiplying the differentiated resolvent equation from the left by $Q$ and from the right by $Q^\top$, we obtain for almost every $t < \tau_{\max}$,
	\begin{equation*}
		\ell_0 \Psi' = \widehat{f} + \bigl(\widehat{M} + iD\bigr)^\top\Psi + \Psi \bigl(\widehat{M} + iD\bigr) - 2\Psi^2 - \ell(0)\Psi - \bigl(\ell' \ast \Psi\bigr).
	\end{equation*}
	Now, all terms except $-2\Psi^2$ are at most linear in $\Psi$, while the quadratic term can now be treated explicitly.
	Since $(A+iB)^2 = A^2 - B^2 + i\bigl(AB + BA\bigr)$, taking real and imaginary parts gives
	\begin{align}
	\ell_0 A' &= \re\widehat{f} + \widehat{M}^\top A + A\widehat{M}-D^\top B - BD - 2A^2 + 2B^2 - \ell(0) A - (\ell' \ast A),\label{eq:complex-real}\\
	\ell_0 B' &= \im\widehat{f} + \widehat{M}^\top B + B\widehat{M} + D^\top A + AD - 2(AB+BA) - \ell(0)B - (\ell' \ast B).\label{eq:complex-imag}
	\end{align}
	All deterministic coefficients in these equations are bounded on $[0,T]$. Now let
	\begin{equation*}
	p(t) := \Tr(A(t)), \quad q(t) := \|B(t)\|^2, \quad E(t) := p(t) + q(t).
	\end{equation*}
	Since $A \succeq 0$, $p \geq 0$. Taking traces in \eqref{eq:complex-real} and using Young's inequality gives a constant $C_T < \infty$ such that
	\begin{equation}\label{eq:p-bound}
	\ell_0 p'(t) \leq C_T \bigl(1 + p(t) + q(t)\bigr) - 2\|A(t)\|^2 + \bigl(|\ell'| \ast p\bigr)(t).
	\end{equation}
	Indeed, the only terms which are not immediately linear in $p$ and $q$ are $-2\Tr(A^2) = -2\|A\|^2$, while the two terms involving $D$ are bounded in absolute value by $C_T\|B\| \leq C_T(1+q)$.
	Next, taking the Frobenius inner product of \eqref{eq:complex-imag} with $2B$ yields
	\begin{equation*}
	2\langle B, -2(AB+BA)\rangle = -8\Tr(AB^2) \leq 0,
	\end{equation*}
	since $A \succeq 0$ and $B^2 \succeq 0$. Moreover,
	\begin{equation*}
	2\bigl|\langle B,D^\top A+AD \rangle\bigr| \leq \|A\|^2+C_Tq,
	\end{equation*}
	and
	\begin{equation*}
	-2\langle B(t), (\ell' \ast B)(t)\rangle \leq \|\ell'\|_{L^1(0,T)} q(t)+\bigl(|\ell'| \ast q\bigr)(t).
	\end{equation*}
	Consequently, after increasing $C_T$ if necessary,
	\begin{equation}\label{eq:q-bound}
	\ell_0 q'(t) \leq C_T\bigl(1+q(t)\bigr) + \|A(t)\|^2 + \bigl(|\ell'| \ast q\bigr)(t).
	\end{equation}
	Adding \eqref{eq:p-bound} and \eqref{eq:q-bound} and dropping the remaining negative term $-\|A\|^2$ gives
	\begin{equation}\label{eq:combined-complex}
	\ell_0 E'(t) \leq C_T(1+E(t)) + \bigl(|\ell'| \ast E\bigr)(t), \quad t < \tau_{\max}.
	\end{equation}
	Integrating \eqref{eq:combined-complex}, using Fubini's theorem, and observing that
	\begin{equation*}
	\int_{0}^{t} (|\ell'| \ast E)(s) \, ds \leq \|\ell'\|_{L^1(0,T)} \int_{0}^{t} E(s) \, ds,
	\end{equation*}
	we obtain
	\begin{equation*}
	E(t) \leq C_T + C_T \int_{0}^{t} E(s) \, ds, \quad t < \tau_{\max}.
	\end{equation*}
	Gronwall's lemma therefore yields
	\begin{equation*}
	\sup_{t < \tau_{\max}} E(t) < \infty.
	\end{equation*}
	Since $A \succeq 0$, $\|A\| \leq \Tr(A) = p$, and hence both $A$ and $B$ remain bounded. Since $Q$ is invertible, $\widetilde{\psi} = Q^{-1}\Psi\bigl(Q^{-1}\bigr)^\top$ remains bounded as well. Thus, $\widetilde{\psi}$ has finite $L^2$-norm on every bounded interval up to $\tau_{\max}$, contradicting the continuation criterion in Theorem~\ref{thm:non-continuable-solutions} if $\tau_{\max} < T$. Hence, $\tau_{\max} = T$. Uniqueness follows from the same theorem, and reversing time completes the proof.
\end{proof}

\begin{lemma}\label{lem:real-global-riccati}
	Let $K = k\Id_d$ with $k$ satisfying Assumptions \ref{K-regularity-condition} and \ref{ass:k-PSD-conditions}. Let $u \in\bS_d^+$, $f \in L^1([0,T],\bS_d^+)$, $M \in\R^{d \times d}$ and $\alpha \in\bS_d^+$. Then
	\begin{equation*}
	 \psi(t) = uk(T-t) + \int_{t}^{T} k(s-t) \bigl(f(s) + M^\top \psi(s) + \psi(s)M - 2\psi(s)\alpha \psi(s)\bigr)\,ds
	\end{equation*}
	has a unique solution $\psi \in L^2([0,T],\bS_d^+)$.
\end{lemma}

\begin{proof}
	After time reversal, Theorem~\ref{thm:non-continuable-solutions} yields a unique non-continuable solution $(\widetilde{\psi}, \tau_{\max})$. Since $k$ is continuous and $f \in L^1$, $\widetilde{\psi}$ has a continuous representative. Write $\widetilde{f}(r) := f(T-r)$. We verify the supporting-normal condition directly. The calculation also establishes quasi-monotonicity, as defined in Appendix~\ref{sec:appendix-riccati}.  Similarly to the proof of Lemma~\ref{lem:GS-exsitence-riccati} define
	\begin{equation*}
	R(X) := M^\top X + XM - 2X\alpha X, \quad X \in \bS_d.
	\end{equation*}
	If $X \preceq Y$, put $D := Y-X \succeq 0$, and let $Z \in \bS_d^+$ satisfy $\Tr(DZ) = 0$. Then $DZ = ZD = 0$ and
	\begin{align*}
		R(Y) - R(X) = M^\top D + DM - 2X\alpha D - 2D\alpha X - 2D\alpha D.
	\end{align*}
	Pairing with $Z$ and using cyclicity of the trace shows that every term on the right-hand side vanishes. Hence, $\langle R(Y) - R(X), Z \rangle = 0$.
	Thus, $R$ is quasi-monotone increasing on the self-dual cone $\bS_d^+$. In particular, if $\widetilde{\psi},Z \in \bS_d^+$ and $\Tr(\widetilde{\psi} Z) = 0$, then
	\begin{equation*}
	\langle \widetilde{f}(r) + R(\widetilde{\psi}), Z \rangle = \langle \widetilde{f}(r), Z\rangle \geq 0 \quad \text{ for a.e. } r.
	\end{equation*}
	Lemma~\ref{lem:volterra-cone-invariance} consequently gives $\widetilde{\psi}(r) \succeq 0$ for all $r < \tau_{\max}$. Having established positive semidefiniteness, it remains to rule out finite-time explosion of the non-continuable solution. The quadratic term in the Riccati equation has the favorable sign $-2\widetilde{\psi}\alpha\widetilde{\psi} \preceq 0$. This suggests comparing $\widetilde{\psi}$ from above with the solution obtained by dropping this negative quadratic term, which is a linear equation that admits a global solution. Let $\Lambda$ denote the unique global solution of the linear Volterra equation
	\begin{equation*}
	\Lambda(t) = uk(t) + \int_{0}^{t} k(t-s)\bigl(f(T-s) + M^\top \Lambda(s) + \Lambda(s)M\bigr)\,ds.
	\end{equation*}
	Global existence follows from \cite[Corollary B.3]{abijaber_Affine_2019} after vectorization. By Lemma~\ref{lem:volterra-cone-invariance}, we have again that $\Lambda(t) \succeq 0$. The difference $D = \Lambda-\widetilde{\psi}$ satisfies, on $[0,\tau_{\max})$,
	\begin{equation*}
	D(t) = \int_{0}^{t} k(t-s)\bigl(M^\top D(s) + D(s)M + 2\widetilde{\psi}(s)\alpha \widetilde{\psi}(s)\bigr)\,ds.
	\end{equation*}
	Since $\widetilde{\psi}\alpha \widetilde{\psi} \succeq 0$, another application of Lemma~\ref{lem:volterra-cone-invariance} yields $D \succeq 0$. Thus, 
	\begin{equation*}
	0 \preceq \widetilde{\psi}(t) \preceq \Lambda(t), \quad t < \tau_{\max}.
	\end{equation*}
	For positive semidefinite matrices this implies $\|\widetilde{\psi}(t)\| \leq \Tr(\widetilde{\psi}(t)) \leq \Tr(\Lambda(t)) \leq \sqrt{d}\|\Lambda(t)\|$. Hence, $\widetilde{\psi}$ cannot blow up on a finite interval. The continuation criterion in Theorem~\ref{thm:non-continuable-solutions} therefore forces $\tau_{\max} = T$. Uniqueness follows from the same theorem.
\end{proof}

\begin{remark}
    The Riccati--Volterra equation above is closely related to the
    fractional Riccati equations appearing in rough affine models.
    To see this, consider first, formally, the fractional kernel
    \begin{equation*}
        k_\alpha(t) = \frac{t^{\alpha-1}}{\Gamma(\alpha)}, \quad \alpha = H + \frac{1}{2} \in \left(\frac{1}{2},1\right).
    \end{equation*}
    Writing the backward Riccati--Volterra equation in terms of the time-to-maturity variable, i.e.
    \begin{equation*}
        \widehat{\psi}(t):=\psi(T-t),
    \end{equation*}
    gives
    \begin{equation*}
        \widehat{\psi}(t) = u_3 k_\alpha(t) + \int_{0}^{t} k_\alpha(t-s) R(T-s,\widehat{\psi}(s))\,ds.
    \end{equation*}
    Since convolution with $k_\alpha$ coincides with the Riemann--Liouville fractional integral $I^\alpha$, this can be equivalently written as
    \begin{equation*}
        D^\alpha\widehat{\psi}(t) = R(T-t,\widehat{\psi}(t)), \quad \bigl(I^{1-\alpha}\widehat{\psi}\bigr)(0) = u_3.
    \end{equation*}
    Equivalently, in the original time variable, this takes the form
    \begin{equation*}
        (D^\alpha\widehat{\psi})(T-t) = R(t,\widehat{\psi}(T-t)), \quad t \in [0,T],
    \end{equation*}
    which is the usual form of the fractional Riccati equation in
    rough affine models.

    The singular fractional kernel $k_\alpha$ is, however, not covered by the existence result above. In our setting we instead consider the shifted fractional kernel
    \begin{equation*}
        k_{\alpha,\varepsilon}(t) = \frac{(t+\varepsilon)^{\alpha-1}}{\Gamma(\alpha)}, \quad \varepsilon > 0.
    \end{equation*}
    Let $L_\varepsilon$ denote its resolvent of the first kind and set
    \begin{equation*}
        D^{k_{\alpha,\varepsilon}}g := \frac{d}{dt}(L_\varepsilon \ast g).
    \end{equation*}
    Then the time-reversed Riccati--Volterra equation
    \begin{equation*}
        \widehat{\psi}(t) = u_3 k_{\alpha,\varepsilon}(t) + \int_{0}^{t} k_{\alpha,\varepsilon}(t-s) R(T-s,\widehat{\psi}(s))\,ds
    \end{equation*}
    is equivalently expressed as
    \begin{equation*}
        D^{k_{\alpha,\varepsilon}} \widehat{\psi}(t) = R(T-t,\widehat{\psi}(t)), \quad (L_\varepsilon*\widehat{\psi})(0) = u_3.
    \end{equation*}
    Or, in the original time variable,
    \begin{equation*}
        \bigl(D^{k_{\alpha,\varepsilon}} \widehat{\psi}\bigr)(T-t) = R(t,\widehat{\psi}(T-t)).
    \end{equation*}
\end{remark}

\newpage

\appendix

\section{Solutions and invariance results for matrix-valued Volterra integral equations}
\label{sec:appendix-riccati}

Fix $T < \infty$, let $k \in \loc^2(\R_+,\R)$, $g : \R_+ \to \bS_d(\C)$ and $p : \R_+ \times \bS_d(\C) \to \bS_d(\C)$, and consider
\begin{equation}
	\psi(t) = g(T-t) + \int_{t}^{T} k(s-t)p(s,\psi(s))\,ds, \quad t \in [0,T].
\end{equation}
With $\widetilde{\psi}(r) = \psi(T-r)$, this is equivalent to
\begin{equation}\label{eq:appenidx-riccati-FE}
	\widetilde{\psi}(t) = g(t) + \int_{0}^{t} k(t-s)p(T-s,\widetilde{\psi}(s))\,ds, \quad t \in [0,T].
\end{equation}

Following \cite[Appendix B]{abijaber_Affine_2019}, a non-continuable solution of \eqref{eq:appenidx-riccati-FE} is a pair $(\widetilde{\psi},\tau_{\max})$, $\tau_{\max} \in (0,\infty]$, such that $\widetilde{\psi} \in \loc^2([0,\tau_{\max}),\bS_d(\C))$ solves the equation on $[0,\tau_{\max})$ and, if $\tau_{\max}<\infty$, then
\begin{equation*}
\|\widetilde{\psi}\|_{L^2(0,\tau_{\max})} = \infty.
\end{equation*}

\begin{theorem}\label{thm:non-continuable-solutions}
	Assume $g \in \loc^2(\R_+,\bS_d(\C))$, $p(\cdot,0) \in \loc^1(\R_+,\bS_d(\C))$, and that for every $T < \infty$ there are $\Pi_T \in L^2([0,T],\R_+)$ and $\Theta_T \geq 0$ such that
	\begin{equation}\label{eq:riccati-existence-growth-con}
	\|p(t,x)-p(t,y)\| \leq \Pi_T(t)\|x-y\| + \Theta_T\|x-y\|(\|x\|+\|y\|)
	\end{equation}
	for $x,y \in \bS_d(\C)$ and $t \leq T$. Then \eqref{eq:appenidx-riccati-FE} has a unique non-continuable solution $(\widetilde{\psi},\tau_{\max})$ with
	\begin{equation*}
	\widetilde{\psi} \in \loc^2([0,\tau_{\max}),\bS_d(\C)).
	\end{equation*}
	If $g$ and $p$ are real-valued, then so is $\widetilde{\psi}$.
\end{theorem}

\begin{proof}
	Let $y = \vc(\widetilde{\psi})$. Since multiplication by the scalar kernel commutes with vectorization, \eqref{eq:appenidx-riccati-FE} becomes
	\begin{equation*}
	y(t) = \vc(g(t)) + \int_{0}^{t} k(t-s)q(s,y(s))\,ds, \quad q(s,z) = \vc\bigl(p(T-s,\vc^{-1}(z))\bigr).
	\end{equation*}
	Vectorization is an isometry for the Frobenius norm, so \eqref{eq:riccati-existence-growth-con} transfers directly to $q$. Thus, \cite[Theorem B.1]{abijaber_Affine_2019} yields a unique non-continuable solution in $\loc^2$. Since $g$ and $p$ take values in the symmetric matrices, the Picard iterates used in the local existence proof remain in the symmetric subspace. Uniqueness therefore implies that the solution is $\bS_d(\C)$-valued. The real-valued statement follows in the same manner.
\end{proof}

The following definition introduces the quasi-monotonicity condition underlying the comparison arguments for matrix Riccati equations; see, in particular, \cite{cuchiero_Affine_2011} for the classical affine setting on $\bS_d^+$. We then formulate a deterministic cone-invariance principle for Volterra equations with nonnegativity-preserving kernels, in the spirit of \cite{alfonsi_Nonnegativity_2025,abijaber_Weak_2026}, and adapted to the finite-dimensional setting required here.

\begin{definition}[Quasi-monotonicity]
    Let $U \subset \bS_d$ be an open set. A function $f : U \to \bS_d$ is called \emph{quasi-monotone increasing}
    if, for all $X,Y \in U$ and $Z \in \bS_d^+$ such that
    \begin{equation*}
        X \preceq Y, \quad \langle X,Z\rangle = \langle Y,Z\rangle,
    \end{equation*}
    it holds that
    \begin{equation*}
        \langle f(X),Z \rangle \leq \langle f(Y),Z \rangle .
    \end{equation*}
\end{definition}

The following invariance argument is inspired by the splitting approach of \cite[Section 3]{alfonsi_Nonnegativity_2025}, and subsequently extended to more general Volterra kernels in \cite[Section 2.1]{abijaber_Weak_2026}. The basic idea is to separate, on each time step, the propagation of the past through the kernel from the local dynamics and to use the nonnegativity-preserving property of the kernel to retain the cone constraint at each step. We adapt this argument here to the deterministic Volterra equation arising in our Riccati-equations.

\begin{lemma}\label{lem:volterra-cone-invariance}
	Let $E$ be a finite-dimensional real Hilbert space, $C\subset E$ a closed convex cone with dual cone $C^\ast$, and let $k:[0,T]\to\R_+$ be continuous, satisfy $k(0)>0$, and preserve nonnegativity in the sense of Definition~\ref{def:nonnegativity-preserving-kernel}. Let $u\in C$, let $F\in L^1([0,T],C)$, and let $G:[0,T] \times E \to E$ be continuous and locally Lipschitz in the second variable, locally uniformly in time. Assume that for a.e. $t \in [0,T]$,
	\begin{equation}\label{eq:cone-inward-condition}
		\langle F(t)+G(t,x),z\rangle \geq 0
	\end{equation}
	for every $x\in C$ and $z\in C^\ast$ satisfying $\langle x,z\rangle = 0$. Then every continuous non-continuable solution of
	\begin{equation}\label{eq:cone-volterra-equation}
		x(t) = k(t)u+\int_{0}^{t} k(t-s)\bigl(F(s)+G(s,x(s))\bigr)\,ds
	\end{equation}
	remains in $C$ on its interval of existence.
\end{lemma}

\begin{proof}
	We first suppose that $G$ is globally Lipschitz in its second variable, uniformly on $[0,T]$. For $N\in\N$, let $0 = t_0<t_1<\cdots<t_N = T$ be the uniform grid with mesh $h = T/N$. We construct a deterministic splitting approximation $(\widehat{x}^N,\xi^N)$ as follows. Initially,
	\begin{equation*}
	\widehat{x}^N(t) = k(t)u,\quad 0 \leq t<t_1.
	\end{equation*}
	Suppose that $\widehat{x}^N$ has been constructed up to $t_j$. For $t\in(t_j,t_{j+1})$, propagate all previous increments only through the kernel and set
	\begin{equation}\label{eq:cone-splitting-between}
	\widehat{x}^N(t) = k(t)u + \sum_{i = 1}^{j}k(t-t_i)\Delta_i^N, \qquad
	\Delta_i^N := \frac{\widehat{x}^N(t_i) - \widehat{x}^N(t_i-)}{k(0)}.
	\end{equation}
	Starting from the left limit $\widehat{x}^N(t_{j+1}-)$, let $\xi^N$ on $[t_j,t_{j+1}]$ solve the ODE
	\begin{equation}\label{eq:cone-splitting-ode}
	\xi^N(s) = \widehat{x}^N(t_{j+1}-) + k(0)\int_{t_j}^{s} \bigl(F(r)+G(r,\xi^N(r))\bigr)\,dr, \quad s \in [t_j,t_{j+1}],
	\end{equation}
	and define $\widehat{x}^N(t_{j+1}):=\xi^N(t_{j+1})$. Thus,
	\begin{equation}\label{eq:cone-splitting-increment}
	\Delta_{j+1}^N = \int_{t_j}^{t_{j+1}} \bigl(F(r)+G(r,\xi^N(r))\bigr)\,dr.
	\end{equation}

	We claim that $\widehat{x}^N(t) \in C$ for all $t \in [0,T]$. We first verify directly that the ODE step in \eqref{eq:cone-splitting-ode} leaves $C$ invariant. Let $\Pi_C$ denote the projection onto $C$ and put $d_C(y) = \|y-\Pi_Cy\|$. For an absolutely continuous solution $\xi$ of \eqref{eq:cone-splitting-ode}, the squared distance is absolutely continuous and, for a.e. $s$,
	\begin{equation*}
	\frac{1}{2}\frac{d}{ds}d_C(\xi(s))^2 = k(0)\bigl\langle \xi(s)-\Pi_C\xi(s),F(s)+G(s,\xi(s))\bigr\rangle.
	\end{equation*}
	Set $p = \Pi_C\xi(s)$ and $n = \xi(s)-p$. The characterization of the projection onto a closed convex cone gives $\langle n,p\rangle = 0$ and $\langle n,y\rangle\leq0$ for every $y\in C$. Hence, $z := -n\in C^\ast$ and $\langle p,z\rangle = 0$. By \eqref{eq:cone-inward-condition},
	\begin{equation*}
	\langle n,F(s)+G(s,p)\rangle\leq0.
	\end{equation*}
	If $L$ is a global Lipschitz constant of $G$ in its second variable, then
	\begin{equation*}
	\frac{1}{2}\frac{d}{ds}d_C(\xi(s))^2
	\leq k(0)L\|n\|^2
	=k(0)L d_C(\xi(s))^2.
	\end{equation*}
	Thus, an ODE trajectory which starts in $C$ has $d_C(\xi(s)) = 0$ throughout the step by Gronwall's lemma. Consequently, whenever $\widehat{x}^N(t_{j+1}-)\in C$, also $\widehat{x}^N(t_{j+1}) \in C$.

	It remains to propagate positivity between grid points. Fix $z \in C^\ast$. With
	\begin{equation*}
	a_0 := \langle u,z\rangle,\qquad
	a_i := \langle\Delta_i^N,z\rangle,\quad i \geq 1,
	\end{equation*}
	we have at every grid point already constructed
	\begin{equation*}
	\sum_{i = 0}^{j}a_i k(t_j-t_i)
	=\langle\widehat{x}^N(t_j),z\rangle \geq 0.
	\end{equation*}
	Since $k$ preserves nonnegativity, Definition~\ref{def:nonnegativity-preserving-kernel} therefore implies
	\begin{equation*}
	\langle\widehat{x}^N(t),z\rangle
	=\sum_{t_i\leq t}a_i k(t-t_i)\geq 0
	\end{equation*}
	until the next grid point. As this holds for every $z \in C^\ast$ and $C$ is a closed convex cone, the characterization
\begin{equation*}
    C = \{x\in E:\langle x,z\rangle\geq 0
    \text{ for all }z\in C^\ast\}
\end{equation*}
implies that $\widehat{x}^N(t) \in C$. Starting from $k(t)u \in C$ on $[0,t_1)$, induction proves the claim on all of $[0,T]$. 
	We next pass to the limit. Let $\omega_k$ denote the modulus of continuity of $k$ on $[0,T]$ and set
	\begin{equation*}
	\varepsilon_N := h+\omega_k(h) + \sup_{I \subset[0,T],\ |I| \leq h} \int_I \|F(s)\| \, ds,
	\end{equation*}
	so that $\varepsilon_N\to0$. The deterministic estimates underlying the splitting approximation in \cite[Proposition 3.2]{alfonsi_Nonnegativity_2025} apply verbatim here (with the stochastic coefficient equal to zero). For completeness, the relevant bounds are as follows. Global Lipschitz continuity of $G$, the representation \eqref{eq:cone-splitting-between}, and Gronwall's lemma first give a constant $C_T$, independent of $N$, such that
	\begin{equation*}
	\sup_{t\leq T}\bigl(\|\widehat{x}^N(t)\|+\|\xi^N(t)\|\bigr) \leq C_T.
	\end{equation*}
	On each interval $[t_j,t_{j+1})$, \eqref{eq:cone-splitting-ode} and the uniform continuity of $k$ then yield
	\begin{equation}\label{eq:cone-splitting-close}
	\sup_{t \leq T} \|\xi^N(t) - \widehat{x}^N(t)\| \leq C_T\varepsilon_N.
	\end{equation}
	Writing $r_N(s) = t_{j+1}$ for $s \in [t_j,t_{j+1})$, equations \eqref{eq:cone-splitting-between}--\eqref{eq:cone-splitting-increment} give, for $t \in [t_m,t_{m+1})$,
	$$\widehat{x}^N(t) = k(t)u + \int_{0}^{t_m} k\bigl(t-r_N(s)\bigr) \bigl(F(s) + G(s,\xi^N(s))\bigr)\,ds.$$
	Comparing this identity with \eqref{eq:cone-volterra-equation}, using \eqref{eq:cone-splitting-close}, $0\leq r_N(s)-s\leq h$, and the absolute continuity of the $L^1$-integral of $F$, we obtain
	$$\sup_{r \leq t} \|\widehat{x}^N(r) - x(r)\| \leq C_T\varepsilon_N +C_T\int_{0}^{t} \sup_{v\leq s}\| \widehat{x}^N(v) - x(v)\| \, ds.$$
	Gronwall's lemma shows that $\widehat{x}^N \to x$ uniformly on $[0,T]$. Since $C$ is closed and every $\widehat{x}^N$ is $C$-valued, $x(t) \in C$ for all $t \in [0,T]$.

It remains to remove the assumption that $G$ is globally Lipschitz. Assume therefore that $G$ is only locally Lipschitz, and fix $T_0$ strictly smaller than the maximal lifetime of the solution $x$.
Since $x$ is continuous on the compact interval $[0,T_0]$, its range $x([0,T_0])$ is bounded. Hence, there exists $R>0$ such that $\|x(t)\| \leq R$ on $t \in [0,T_0]$. Choose a Lipschitz cut-off function $\chi: E \to [0,1] $ such that $\chi(y) = 1$ for $\|y\| \leq R$, and $\chi(y) = 0$ outside some larger ball. Define also the truncated map $\widetilde{G}(t,y) := \chi(y)G(t,y)$.
Since $G$ is locally Lipschitz and $\widetilde{G}$ vanishes outside a bounded set, $\widetilde{G}$ is globally Lipschitz in the second variable.
We next check that the inward-pointing condition is preserved by this truncation. Let $y\in C$ and $z\in C^\ast$ satisfy $\langle y,z\rangle = 0$. Then
\begin{equation*}
	\langle F(t) + \widetilde{G}(t,y),z\rangle = \chi(y)\langle F(t) + G(t,y),z\rangle + (1-\chi(y))\langle F(t),z\rangle.
\end{equation*}
The first term on the right-hand side is nonnegative by \eqref{eq:cone-inward-condition}, while the second is nonnegative since $F(t)\in C$ and $z\in C^\ast$. Thus, $\langle F(t)+\widetilde{G}(t,y),z\rangle \geq 0$. Moreover, $\chi(x(t)) = 1$ for all $t\in[0,T_0]$, and therefore $\widetilde{G}(t,x(t)) = G(t,x(t))$ on $[0,T_0]$. Hence, the original solution $x$ also solves the Volterra equation with $\widetilde{G}$ on this interval. Since $\widetilde{G}$ is globally Lipschitz, the first part of the proof applies and yields
\begin{equation*}
    x(t) \in C, \quad t \in [0,T_0].
\end{equation*}
Since $T_0$ was arbitrary, the conclusion holds throughout the maximal interval.
\end{proof}

\begin{remark}\label{rem:matrix-quasimonotone}
	For $E = \bS_d$ equipped with $\langle x,y\rangle = \Tr(xy)$, the cone $\bS_d^+$ is self-dual. If $x,z \in \bS_d^+$ and $\Tr(xz) = 0$, then $xz = zx = 0$. Hence, the map
	\begin{equation*}
		x \mapsto M^\top x+xM-2x\alpha x
	\end{equation*}
	is quasi-monotone increasing on $\bS_d^+$. This is exactly the geometric ingredient used for matrix Riccati equations in \cite{cuchiero_Affine_2011}. The difference in the present Volterra setting is that quasi-monotonicity has to be combined with the nonnegativity-preserving property of $k$, which is what Lemma~\ref{lem:volterra-cone-invariance} provides.

	For the complex Riccati equation we work in the real Hilbert space $E = \bS_2\times\bS_2$ and identify $\bS_2^{+} + i\bS_2$ with the cone $C = \bS_2^+ \times \bS_2$. Its dual cone is $C^\ast = \bS_2^+\times\{0\}$. Consequently, the supporting-normal condition of Lemma~\ref{lem:volterra-cone-invariance} is exactly the inequality
	\begin{equation*}
	\langle \re R(P+iY), Z\rangle \geq 0, \quad P,Z \succeq 0, \quad \Tr(PZ) = 0,
	\end{equation*}
	verified in Lemma~\ref{lem:GS-exsitence-riccati}.
\end{remark}

\nocite{*}
\bibliographystyle{agsm}
\bibliography{reference.bib}
\end{document}